\documentclass[11pt]{article}
\usepackage{amsmath,amsthm, amssymb, amsfonts,accents,nccmath}
\usepackage{enumitem}
\usepackage{array}
\usepackage[toc,page]{appendix}
\usepackage{mathrsfs}
\usepackage[english]{babel}
\usepackage{dsfont}
\usepackage{booktabs}
\usepackage[linesnumbered,commentsnumbered,ruled,vlined]{algorithm2e}
\usepackage{tikz}
\usepackage{tikz-network}
\usepackage[utf8]{inputenc}
\usepackage[english]{babel}
\usepackage{caption}
\usepackage{subcaption}
\usepackage{float}
\usepackage{lmodern}
\usepackage{bera}

\usepackage{hyperref}
\hypersetup{
    colorlinks=true,
    linkcolor=blue,
    filecolor=darkmagenta,      
    urlcolor=cyan,
    citecolor=red
} 
\makeatletter
\def\BState{\State\hskip-\ALG@thistlm}
\makeatother
\numberwithin{equation}{section}

\newtheorem{theorem}{Theorem}[section]
\newtheorem{cor}[theorem]{Corollary}
\newtheorem{lem}[theorem]{Lemma}
\newtheorem{prop}[theorem]{Proposition}
\newtheorem{defn}[theorem]{Definition}
\newtheorem{rem}[theorem]{Remark}

\newtheorem{con}[theorem]{Condition}

\newcommand{\ds}{\displaystyle}

\makeatletter
\renewcommand*\env@matrix[1][*\c@MaxMatrixCols c]{%
	\hskip -\arraycolsep
	\let\@ifnextchar\new@ifnextchar
	\array{#1}}
\makeatother

\title{A $q$-analogue of distance matrix of bi-block graphs}
\author{Joyentanuj Das\footnote{Department of Mathematics, College of Engineering and Technology, SRM Institute of Science and Technology, Kattankulathur, Chennai, 603203, India. \indent  Email: joyentanuj@gmail.com,  joyentad@srmist.edu.in}}

\date{}

\begin{document}

\maketitle

\begin{abstract}
A $q$-analogue of the distance matrix, referred to as the \emph{$q$-distance matrix}, is obtained from the distance matrix by replacing each nonzero entry $\alpha$ with the sum $1+q+\cdots+q^{\alpha-1}$. This notion was introduced independently by Bapat, Lal, and Pati~\cite{Ba-Lal-Pati}, and by Yan and Yeh~\cite{Yan}. A connected graph is called a \emph{bi-block graph} if each of its blocks is a complete bipartite graph. In this paper, we derive explicit formulas for the determinant and the inverse of the $q$-distance matrix of bi-block graphs. These results both generalize the corresponding formulas for the distance matrix of bi-block graphs obtained in~\cite{Hou3} and extend the results for block graphs in~\cite{Xing} to the class of bi-block graphs.

\end{abstract}


\section{Introduction  and Motivation}
Let $G = (V(G), E(G))$ be a finite, simple, connected graph, where $V(G)$ denotes the vertex set and $E(G) \subseteq V(G) \times V(G)$ the edge set. For brevity, we write $G = (V, E)$ whenever the context is clear. We use $i \sim j$ to indicate adjacency between vertices $i, j \in V$.  

Before proceeding further, we first set up some notations that will be frequently used throughout this article. Let $\mathbf{I}_n$, $\mathds{1}_n$, and $\mathbf{e}_i$ denote, respectively, the identity matrix, the column vector of all ones, and the column vector having $1$ in the $i^{\text{th}}$ entry with all other entries equal to zero. 
In addition, $\mathbf{J}_{m \times n}$ represents the $m \times n$ matrix of all ones, and when $m=n$ we simply write $\mathbf{J}_m$. We denote by $\mathbf{0}_{m \times n}$ the zero matrix of order $m \times n$, and abbreviate this as $\mathbf{0}$ whenever the order is clear from context. For any matrix $A$, we write $A^T$ to indicate its transpose.

A graph is called \emph{bipartite} if its vertex set admits a partition into two subsets $X \cup Y$ such that $E \subseteq X \times Y$. A bipartite graph is said to be \emph{complete bipartite graph}, denoted by $K_{s,t}$, when $|X| = s$, $|Y| = t$, and every vertex of $X$ is adjacent to every vertex of $Y$.  

Equipped with the shortest--path distance, a connected graph $G$ can be regarded as a metric space. Specifically, the metric $d(i,j)$ is defined as the length of the shortest path between vertices $i$ and $j$, with $d(i,i)=0$. For later use, we recall the notions of the \emph{distance matrix} and the \emph{Laplacian matrix} associated with a graph $G$. 

Let $G$ be a graph on $n$ vertices. The \emph{distance matrix} of $G$ is the $n \times n$ matrix $D(G) = [d_{ij}]$, where $d_{ij} = d(i,j)$. The \emph{Laplacian matrix} of $G$ is the $n \times n$ matrix $L(G) = [l_{ij}]$, defined by  
\[
l_{ij} =
\begin{cases}
	\delta_i & \text{if } i=j, \\
	-1 & \text{if } i \neq j \text{ and } i \sim j, \\
	0 & \text{otherwise},
\end{cases}
\]
where $\delta_i$ denotes the degree of vertex $i$. It is well known that $L(G)$ is symmetric and positive semidefinite. Moreover, the all-ones vector $\mathds{1}$ is an eigenvector of $L(G)$ corresponding to the smallest eigenvalue $0$, so that $L(G)\mathds{1} = \mathbf{0}$ and $\mathds{1}^T L(G) = \mathbf{0}$ (see~\cite{Bapat}).  

The theory of distance matrices has a long and rich development. The subject traces back to the classical work of Graham and Pollack~\cite{Gr1}, who established that for any tree $T$ with $n$ vertices,  
\[
\det D(T) = (-1)^{n-1}(n-1)2^{\,n-2}.
\]
Interestingly, this determinant depends only on the order of the tree, not on its structure. Their result immediately attracted considerable attention, leading to a series of extensions and generalizations (see, e.g., \cite{Ba-Kr-Neu,Bapat1,Gr3,Gr2}).  

In a subsequent paper, Graham and Lov\'asz~\cite{Gr2} derived an explicit expression for the inverse of the distance matrix of a tree $T$:  
\[
D(T)^{-1} = -\tfrac{1}{2} L(T) + \dfrac{1}{2(n-1)} \tau \tau^T,
\qquad \text{where } \tau = (2-\delta_1,\, 2-\delta_2,\, \dots,\, 2-\delta_n)^T.
\]
This identity represents the inverse of the distance matrix in terms of the Laplacian matrix and opened an active line of research on extending the formula to broader graph classes. Subsequent contributions have established analogous results for weighted trees~\cite{Ba-Kr-Neu}, bidirected trees~\cite{Ba-Lal-Pati}, block graphs~\cite{Bp3}, completely positive graphs~\cite{JD}, multi-block graphs~\cite{JD1}, weighted cactoid-type digraphs~\cite{JD2}, cactoid digraphs~\cite{Hou1}, cycle--clique graphs~\cite{Hou2}, bi-block graphs~\cite{Hou3}, distance well-defined graphs~\cite{Zhou1}, and weighted cactoid digraphs~\cite{Zhou2}.  

For an indeterminate $q$, Bapat, Lal, and Pati~\cite{Ba-Lal-Pati}, and independently Yan and Yeh~\cite{Yan}, introduced a $q$-analogue of the distance matrix of a graph $G$, denoted by $\mathscr{D} = \mathscr{D}(G) = (\mathscr{D}_{ij})$, defined as  
\[
\mathscr{D}_{ij} =
\begin{cases}
	[\alpha] = 1+q+\cdots+q^{\alpha-1}, & \text{if } d_{ij} = \alpha \ge 1, \\[6pt]
	0, & \text{if } i=j.
\end{cases}
\]
The matrix $\mathscr{D}$ is called the \emph{$q$-distance matrix} of $G$. In particular, when $q=1$, it reduces to the standard distance matrix.  

A vertex $v$ of a graph $G$ is a \emph{cut-vertex} if the deletion of $v$ disconnects $G$. A \emph{block} of $G$ is a maximal connected subgraph with no cut-vertex. In~\cite{Gr3}, Graham, Hoffman, and Hosoya derived a formula expressing the determinant of the distance matrix of a graph (and of strongly connected digraphs) in terms of its blocks, thereby explaining the classical formula for trees. The $q$-distance matrix was subsequently employed by Yan and Yeh~\cite{Yan} to obtain a $q$-analogue of Graham and Pollack’s determinant formula for trees.  

In~\cite{Bp3}, Bapat and Sivasubramanian computed the determinant and inverse of the distance matrix for block graphs, i.e., graphs in which every block is a complete graph. Later, Xing and Du~\cite{Xing} extended this result to the $q$-distance matrix of block graphs. In~\cite{Hou3}, Hou and Sun obtained analogous formulas for the distance matrix of bi-block graphs, i.e., graphs in which every block is a complete bipartite graph.  

In the present work, we extend these results further by deriving explicit formulas for the determinant and the inverse of the $q$-distance matrix of bi-block graphs. We note that the theory of $q$-distance matrices has been the subject of active investigation in recent years; see, for instance, trees with matrix weights~\cite{Barik}, $3$-hypertrees~\cite{Siva}, $q$-analogue of Graham, Hoffman and Hosoya’s theorem~\cite{Siva1} and trees and two quantities relating to permutations~\cite{Yan}.

\section{Notations and Preliminaries}\label{sec:notations}
\subsection{Preliminary results}
Now we will recall a few preliminary results from matrix theory and graph theory which will useful for subsequent sections. Let $\mathscr{M}$ be an $n\times n$ matrix partitioned as
\begin{equation}\label{eqn:M}
	\mathscr{M}= \left[
	\begin{array}{cc}
		A& B \\
		C & D
	\end{array}
	\right],
\end{equation}
where $A$ and $D$ are square matrices. If $A$ is nonsingular, then the Schur complement of $A$ in $\mathscr{M}$ is defined to be the matrix $D-CA^{-1}B$. Simillarly, if $D$ is nonsingular, then the Schur complement of $D$ in $\mathscr{M}$ is defined to be the matrix $A - BD^{-1}C$. The next result uses Schur complement to find the determinant of a $2 \times 2$ block matrix.

\begin{prop}\label{prop:detblock}\cite{Zhang1}
	Let $\mathscr{M}$ be a nonsingular matrix partitioned as in Eqn~(\ref{eqn:M}). If $A$ is square and invertible, then $\det(\mathscr{M}) = \det(A) \det(D-CA^{-1}B)$ and if $D$ is square and invertible, then $\det(\mathscr{M}) = \det(D) \det(A - BD^{-1}C)$.
\end{prop}

The next result gives us the inverse of a partitioned matrix using Schur complement, whenever the matrix is invertible.
\begin{prop}\label{prop:schur}\cite{Zhang1}
	Let $\mathscr{M}$ be a nonsingular matrix partitioned as in Eqn~(\ref{eqn:M}). If $A$ is square and invertible, then
	\[
	\mathscr{M}^{-1} = \begin{bmatrix}[cc]
	A^{-1} + A^{-1} B (D-CA^{-1}B)^{-1} C A^{-1} & -A^{-1} B (D-CA^{-1}B)^{-1}\\
	-(D-CA^{-1}B)^{-1} C A^{-1} & (D-CA^{-1}B)^{-1}
	\end{bmatrix}
	\]
	and if $D$ is square and invertible, then
	\[
	\mathscr{M}^{-1} = \begin{bmatrix}[cc]
		(A - BD^{-1}C)^{-1} & -(A - BD^{-1}C)^{-1}BD^{-1}\\
		-D^{-1}C(A - BD^{-1}C)^{-1} & D^{-1} + D^{-1}C(A - BD^{-1}C)^{-1}BD^{-1}
	\end{bmatrix}
	\]
\end{prop}

Now we state a result without proof that gives the determinant and inverse for matrix of the  form $a\mathbf{I}+b\mathbf{J}$.
\begin{lem}\label{lem:aI+bJ}
	Let $n \geq 2$ and $\mathbf{J}_n$, $\mathbf{I}_n$ be matrices as defined before. For $a \neq 0$, the determinant of $a\mathbf{I}_n+b\mathbf{J}_n$ is given by $a^{n-1}(a+nb)$. Moreover, the matrix is invertible if and only if $a+nb \neq 0$ and the inverse is given by $$(a\mathbf{I}_n + b\mathbf{J}_n)^{-1} = \frac{1}{a} \left(\mathbf{I}_n - \frac{b}{a+nb} \mathbf{J}_n\right).$$
\end{lem}

The concept for cofactor for $q$-distance matrix introduced by Li, et.al. in \cite{Li}. Let $G$ be a graph (strongly connected digraphs) and $\mathscr{D}$ be its $q$-distance matrix which has the following form:
\[ \mathscr{D} = 
\begin{bmatrix}[c|ccc] 
	0 & [\alpha_1] & \cdots & [\alpha_{n-1}] \\ 
	\hline [\beta_1] & & & \\ 
	\vdots & &\mathscr{D}_1 &\\ 
	[\beta_{n-1}] & & & 
	\end{bmatrix} 
\]

Let \( \xi(\mathscr{D}) \) denote the cofactor in position \((1, 1)\) of the matrix obtained by subtracting the first row from all other rows, then \( q^{\alpha_i} \) times the first column from the \((i + 1)\)th column of \( \mathscr{D} \) for \( i = 1, \ldots, n - 1 \). Observe that 
\begin{equation}\label{eqn:D-M}
	\xi(\mathscr{D}) = \det(\mathscr{D}_1 - M),
\end{equation}
where \( M \) is the \((n - 1) \times (n - 1)\) matrix with \([\beta_i + \alpha_j]\) as its \((i, j)\) entry (since \([\alpha + \beta] = q^\beta[\alpha] + [\beta] = [\alpha] + q^\alpha[\beta]\)).

Next, we state the result from \cite{Li} which will help us in finding the determinant and cofactor for a general bi-block graph.
\begin{theorem}[\cite{Li}]\label{thm:det-cof}
	If $G$ is a graph (strongly connected digraph) with blocks $G_1,G_2, \cdots,G_r$ and $\mathscr{D}(G)$ be the $q$-distance matrix of $G$, then
	\[
	\xi(\mathscr{D}(G)) = \prod_{i=1}^{r} \xi(\mathscr{D}(G_i))
	\]
	and
	\[
	\det(\mathscr{D}(G))  = \sum_{i=1}^r \det(\mathscr{D}(G_i)) \prod_{j \ne i} \xi(\mathscr{D}(G_j)).
	\]
\end{theorem}

\subsection{Notations for bi-block graphs}
Let $G$ be an \( \mathbf{n} \)-vertex bi-block graph with $r$ blocks $K_{m_i,n_i}$, $i=1,2,\cdots,r$, where the vertex set of $K_{m_i,n_i}$ is partitioned into $X_i \cup Y_i$, with $|X_i|=m_i$ and $|Y_i|=n_i$. Then $\mathbf{n} = \sum_{i=1}^r (m_i+n_i)-r+1$. The block degree of a vertex $v$ in $G$, denoted by $\hat{d}_G(v)$, is the number of blocks containing $v$. 

\begin{con}\label{con}
Let $G$ be an \( \mathbf{n} \)-vertex bi-block graph with $r$ blocks $K_{m_i,n_i}$, $i=1,2,\cdots,r$. Let $q$ be an indeterminate such that such that $q \ne -1$ and  $q^2(m_i-1)(n_i-1) \ne 1$ for all $1 \le i \le r$.
\end{con}

\begin{con}\label{con-1}
	Let $G$ be an \( \mathbf{n} \)-vertex bi-block graph with $r$ blocks $K_{m_i,n_i}$, $i=1,2,\cdots,r$. Let $q$ be an indeterminate such that such that $q \ne -1$ and  $(q+1)^2(m_i-1)(n_i-1) \ne m_in_i$ for all $1 \le i \le r$.
\end{con}

Under the assumption of Condition~\ref{con}, the following definitions are well defined. Let $\mathbf{x}_G$ and $\mathbf{y}_G$ be two $\mathbf{n}$-dimensional column vector defined as follows:
\begin{equation}\label{eqn:x-0}
	\mathbf{x}_G(v) = \sum_{v \in X_i} \frac{q(n_i-1)-1}{(q+1)(q^2(m_i-1)(n_i-1)-1)} + \sum_{v \in Y_i} \frac{q(m_i-1)-1}{(q+1)(q^2(m_i-1)(n_i-1)-1)} - (k-1)
\end{equation}
and
\begin{equation}\label{eqn:y-0}
	\mathbf{y}_G(v) = \sum_{v \in X_i} \frac{n_i-1}{q^2(m_i-1)(n_i-1)-1} + \sum_{v \in Y_i} \frac{m_i-1}{q^2(m_i-1)(n_i-1)-1} - (k-1),
\end{equation}
where $\hat{d}_G(v) = k$. Let $\mathbf{A}_G = (a_{ij})$ be an $\mathbf{n} \times \mathbf{n}$ weighted adjacency matrix of $G$ defined as follows:
\begin{equation}\label{eqn:A}
	a_{ij} = \begin{cases}
		\dfrac{1}{q^2(m_i-1)(n_i-1)-1}	& \text{ if } i\sim j \text{ and } i,j \in K_{m_i,n_i},\\
		0 & \text{otherwise}
	\end{cases}
\end{equation}
and $\mathbf{B}_G = (b_{ij})$ be an $\mathbf{n} \times \mathbf{n}$ weighted adjacency matrix of $\bar{G}$, the complement of $G$ defined as follows:
\begin{equation}\label{eqn:B}
	b_{ij} = \begin{cases}
		\dfrac{n_i-1}{q^2(m_i-1)(n_i-1)-1}	& \text{ if } i \neq j, i \not\sim j \text{ and } i,j \in X_i,\\
		\dfrac{m_i-1}{q^2(m_i-1)(n_i-1)-1}	& \text{ if } i \neq j, i \not\sim j \text{ and } i,j \in Y_i,\\
		0 & \text{otherwise}.
	\end{cases}
\end{equation}

\subsection{Notations used in induction}\label{sec:ind}
The induction on the number of blocks plays a significant role in our analysis. 
The relation between a bi-block graph with $r$ blocks and a subgraph, which is also a bi-block graph but with $r-1$ blocks, is crucial for our proofs. 
During the application of induction, we will employ the notations introduced below:

As above, assume that $G$ is an $\mathbf{n}$-vertex bi-block graph with $r$ blocks $K_{m_i,n_i}$, $i=1,2,\ldots,r$, where the vertex partition of each block is $X_i \cup Y_i$, with $|X_i| = m_i$ and $|Y_i| = n_i$. 
Without loss of generality, we take $K_{m_r,n_r} = K_{s,t}$ to be a leaf block of $G$ (the block containing exactly one vertex of block degree at least $2$), \textit{i.e.}, $m_r = s$ and $n_r = t$. Let $H$ be the subgraph of $G$ induced by 
\[
V(K_{m_1,n_1}) \cup V(K_{m_2,n_2}) \cup \cdots \cup V(K_{m_{r-1},n_{r-1}}),
\]
and denote $|V(H)| = \mathbf{m}$. 
Thus, $\mathbf{n} = \mathbf{m} + (m_r + n_r - 1)$. Let $v_{\mathbf{m}}$ be the cut-vertex of the leaf block $K_{m_r,n_r} = K_{s,t}$, where $v_{\mathbf{m}} \in X_r$. 
From this point onward, we shall adopt the notations $\mathscr{D}, \mathbf{A}, \mathbf{B}, \mathbf{x}$, and $\mathbf{y}$ for the bi-block graph $G$, and the corresponding notations $\hat{\mathscr{D}}, \hat{\mathbf{A}}, \hat{\mathbf{B}}, \hat{\mathbf{x}}$, and $\hat{\mathbf{y}}$ for its subgraph $H$.

If $v_i \neq v_{\mathbf{m}}$ be a vertex in $H$ and $v_j \neq v_{\mathbf{m}}$ be a vertex in $K_{m_r,n_r} = K_{s,t} $, then we have
$$
d(v_i,v_j) = \begin{cases}
	d(v_i,v_{\mathbf{m}}) + 2 	& \text{ if } v_j \in X_r,\\
	d(v_i,v_{\mathbf{m}}) + 1	& \text{ if } v_j \in Y_r.
\end{cases}
$$ Thus, in the case of $q$-distance matrix, we have
$$
\mathscr{D}_{ij} = \begin{cases}
	1+q+\cdots+q^{d(v_i,v_{\mathbf{m}}) + 2} = 1+q+q^2 \hat{\mathscr{D}}_{i \mathbf{m}} 	& \text{ if } v_j \in X_r,\\
	1+q+\cdots+q^{d(v_i,v_{\mathbf{m}}) + 1} = 1+q \hat{\mathscr{D}}_{i \mathbf{m}} 	& \text{ if } v_j \in Y_r.
\end{cases}
$$
Let $E_0$, $E_1$, and $E_2$ denote the $\mathbf{m} \times \mathbf{m}$, $\mathbf{m} \times (s-1)$, and $\mathbf{m} \times t$ matrices, respectively, whose last row consists entirely of ones and all remaining entries are zero. Define $E_{\mathbf{m}\mathbf{m}}$ as the $\mathbf{m} \times \mathbf{m}$ matrix with a single nonzero entry equal to $1$ at the $(\mathbf{m},\mathbf{m})$ position. With these conventions, the matrices $\mathscr{D}$, $\mathbf{A}$, and $\mathbf{B}$ can be expressed in the following block form:
$$
\mathscr{D} = \begin{bmatrix}
	\hat{\mathscr{D}} & (q+1)\mathbf{J}_{\mathbf{m} \times (s-1)} + q^2 \hat{\mathscr{D}} E_1 & \mathbf{J}_{\mathbf{m} \times t} + q \hat{\mathscr{D}} E_2\\
	(q+1)\mathbf{J}_{(s-1) \times \mathbf{m}} + q^2 E_1^T \hat{\mathscr{D}} & (q+1) (\mathbf{J}_{s-1} - \mathbf{I}_{s-1}) & \mathbf{J}_{(s-1)\times t}\\
	\mathbf{J}_{t \times \mathbf{m}} +q E_2^T \hat{\mathscr{D}} & \mathbf{J}_{t \times (s-1)} & (q+1) (\mathbf{J}_{t} - \mathbf{I}_{t})
\end{bmatrix},
$$
$$
\mathbf{A} = \begin{bmatrix}
	\hat{\mathbf{A}} & \mathbf{0}_{\mathbf{m} \times (s-1)} & \dfrac{1}{q^2(s-1)(t-1)-1} E_2\\
	 \mathbf{0}_{(s-1) \times \mathbf{m}} & \mathbf{0}_{s-1} & \dfrac{1}{q^2(s-1)(t-1)-1} \mathbf{J}_{(s-1)\times t}\\
	 \dfrac{1}{q^2(s-1)(t-1)-1} E_2^T & \dfrac{1}{q^2(s-1)(t-1)-1} \mathbf{J}_{t \times (s-1)} & \mathbf{0}_t
\end{bmatrix},
$$
and
$$
\mathbf{B} = \begin{bmatrix}
	\hat{\mathbf{B}} & \dfrac{t-1}{q^2(s-1)(t-1)-1} E_1  & \mathbf{0}_{\mathbf{m} \times t}\\
	\dfrac{t-1}{q^2(s-1)(t-1)-1} E_1^T  & \dfrac{t-1}{q^2(s-1)(t-1)-1} (\mathbf{J}_{s-1} - \mathbf{I}_{s-1}) &  \mathbf{0}_{(s-1)\times t}\\
	\mathbf{0}_{t \times \mathbf{m}} & \mathbf{0}_{t \times (s-1)} & \dfrac{s-1}{q^2(s-1)(t-1)-1} (\mathbf{J}_{t} - \mathbf{I}_{t})
\end{bmatrix}.
$$
Next, we express $\mathbf{x}$ and $\mathbf{y}$ in terms of $\hat{\mathbf{x}}$ and $\hat{\mathbf{y}}$:
\begin{equation}\label{eqn:x}
	\mathbf{x}(v) = \begin{cases}
		\hat{\mathbf{x}}(v)	& \text{ if } v \in H \text{ and } v \neq v_{\mathbf{m}},\\
		\hat{\mathbf{x}}(v_{\mathbf{m}}) +  \dfrac{q(t-1)-1}{(q+1)(q^2(s-1)(t-1)-1)} -1	& \text{ if }  v = v_{\mathbf{m}},\\
		\dfrac{q(t-1)-1}{(q+1)(q^2(s-1)(t-1)-1)} 	& \text{ if } v \in X_r \text{ and } v \neq v_{\mathbf{m}},\\
		\dfrac{q(s-1)-1}{(q+1)(q^2(s-1)(t-1)-1)}	& \text{ if } v \in Y_r,
	\end{cases}
\end{equation}
and
\begin{equation}\label{eqn:y}
	\mathbf{y}(v) = \begin{cases}
		\hat{\mathbf{y}}(v)	& \text{ if } v \in H \text{ and } v \neq v_{\mathbf{m}},\\
		\hat{\mathbf{y}}(v_{\mathbf{m}}) +  \dfrac{t-1}{q^2(s-1)(t-1)-1} -1	& \text{ if }  v = v_{\mathbf{m}},\\
		\dfrac{t-1}{q^2(s-1)(t-1)-1} 	& \text{ if } v \in X_r \text{ and } v \neq v_{\mathbf{m}},\\
		\dfrac{s-1}{q^2(s-1)(t-1)-1}	& \text{ if } v \in Y_r.
	\end{cases}
\end{equation}
Before proceeding further we state some basic properties without proof that will be used often in the future proof.
\begin{rem}\label{rem:identities}
Let $E_1, E_2$ and $\mathbf{J}$ be the matrices as defined before, then the following identities hold:
	\begin{itemize}
		\item[1.] $E_1 \mathbf{J}_{s-1} = (s-1)E_1, E_2 \mathbf{J}_{t} = t E_2.$
		\item[2.] $E_1 \mathbf{J}_{(s-1) \times t} = (s-1) E_2, E_2 \mathbf{J}_{t \times (s-1)} = t E_1.$
		\item[3.] $\mathbf{J}_{(s-1) \times \mathbf{m}}E_{\mathbf{m} \mathbf{m}} = E_1^T, \mathbf{J}_{t \times \mathbf{m}}E_{\mathbf{m} \mathbf{m}} = E_2^T$.
	\end{itemize}
\end{rem}

\section{Determinant and Inverse of $\mathscr{D}(K_{s,t})$}
We start the section by finding the determinant of the $q$-distance matrix of the complete bipartite graph $K_{s,t}$.

\begin{theorem}\label{thm:D-det-single}
	Let $G$ be the complete bipartite graph $K_{s,t}$ and $\mathscr{D}(K_{s,t})$ be the $q$-distance matrix of $G$. Let $q \neq -1$, then 
	$$\det(\mathscr{D}(K_{s,t}))= (-1)^{s+t-2}(q+1)^{s+t-2} \left[(q+1)^2(s-1)(t-1)-st\right].$$
\end{theorem}

\begin{proof}
	The $q$-distance matrix of $K_{s,t}$ has the following block form:
	$$
	\mathscr{D}(K_{s,t}) = \begin{bmatrix}
		(q+1)(\mathbf{J}_{s} -\mathbf{I}_{s}) &  \mathbf{J}_{s\times t}\\
		 \mathbf{J}_{t \times s} & (q+1)(\mathbf{J}_{t} - \mathbf{I}_{t})
	\end{bmatrix}.
	$$ Next, we use Proposition~\ref{prop:detblock} to find the $\det(\mathscr{D})$. Since, $q+1 \neq 0$, using Lemma~\ref{lem:aI+bJ}, we have
	\[
	\left[(q+1)(\mathbf{J}_{s} - \mathbf{I}_{s})\right]^{-1} = \frac{1}{q+1}\left(\frac{1}{s-1}\mathbf{J}_s - \mathbf{I}_s\right) .
	\]
	Thus, 
	\begin{equation*} 
		\begin{split}
			\det(\mathscr{D}(K_{s,t})) & = \det((q+1)(\mathbf{J}_{s} - \mathbf{I}_{s})) \times \det \left[(q+1)(\mathbf{J}_{t} - \mathbf{I}_{t}) -  \mathbf{J}_{t \times s} \times \left((q+1)(\mathbf{J}_{s} - \mathbf{I}_{s}) \right)^{-1} \times  \mathbf{J}_{s\times t} \right]\\
			& = \det(q^2 \mathbf{J}_{s} - (q^2-1) \mathbf{I}_{s}) \times \det \left[(q+1)(\mathbf{J}_{t} - \mathbf{I}_{t}) -  \mathbf{J}_{t \times s} \times \left(\frac{1}{q+1}\left(\frac{1}{s-1}\mathbf{J}_s - \mathbf{I}_s\right) \right) \times \mathbf{J}_{s\times t} \right]\\
			& = \det(q^2 \mathbf{J}_{s} - (q^2-1) \mathbf{I}_{s}) \times \det \left[(q+1)(\mathbf{J}_{t} - \mathbf{I}_{t}) -   \frac{s}{(q+1)(s-1)}\mathbf{J}_t \right]\\
			& = \det(q^2 \mathbf{J}_{s} - (q^2-1) \mathbf{I}_{s}) \times \det \left[\left(q+1-\frac{s}{(q+1)(s-1)}\right)\mathbf{J}_t - (q+1)\mathbf{I}_{t} \right]\\
			&= (-1)^{s-1}(s-1)(q+1)^s \times (-1)^{t-1} (q+1)^{t-1} \left[ (q+1)t - \frac{st}{(q+1)(s-1)} - (q+1)\right]\\
			&= (-1)^{s+t-2} (q+1)^{s+t-2} \left[ (q+1)^2(s-1)(t-1) - st\right].
		\end{split}
	\end{equation*}
\end{proof}

\begin{rem}
	Note that the determinant of the $q$-distance matrix of the complete bipartite graph $K_{s,t}$ is non-zero if and only if $q \ne 1$ and $(q+1)^2(s-1)(t-1) \ne st$. It was proved by Hou and Sun in~\cite{Hou3}, that the determinant of the distance matrix of the complete bipartite graph $K_{s,t}$ is $(-2)^{s+t}(3st-4s-4t+4)$, which matches with our result when we substitute $q=1$.
\end{rem}

Next, we find the cofactor at the $(1,1)$ position of the matrix $\mathscr{D}(K_{s,t})$. Using the definition of $M$ from Equation~\eqref{eqn:D-M}, we have the following matrix for the case when $G$ is the complete bipartite graph $K_{s,t}$:
\[
M(K_{s,t}) = \begin{bmatrix}[cc]
	[4]\mathbf{J}_s & [3]\mathbf{J}_{s \times t}\\
	[3]\mathbf{J}_{t \times s} & [2]\mathbf{J}_t
\end{bmatrix}
\]
and hence the matrix $(\mathscr{D}_1 - M)(K_{s,t})$ is given by
\[
(\mathscr{D}_1 - M)(K_{s,t}) = \begin{bmatrix}[cc]
	-q^2(q+1)\mathbf{J}_{s-1} - (q+1)\mathbf{I}_{s-1} & -q(q+1)\mathbf{J}_{(s-1) \times t}\\
	-q(q+1) \mathbf{J}_{t \times (s-1)} & -(q+1)\mathbf{I}_t
\end{bmatrix}.
\]
In the next Lemma, we find \( \xi(\mathscr{D}((K_{s,t}))) \) for the complete bipartite graph $K_{s,t}$ using the matrix $(\mathscr{D}_1 - M)(K_{s,t})$.

\begin{lem}\label{lem:cof-single}
	Let $G$ be the complete bipartite graph $K_{s,t}$ and $\mathscr{D}(K_{s,t})$ be the $q$-distance matrix of $G$. Let $q \neq -1$, then 
	\[
	\xi(\mathscr{D}(K_{s,t})) = (-1)^{s+t-1} (q+1)^{s+t-1} \left[q^2(s-1)(t-1)-1\right].
	\]
\end{lem}
\begin{proof}
	Since, \( \xi(\mathscr{D}(K_{s,t})) = \det((\mathscr{D}_1 - M)(K_{s,t})) \), we use the Schur complement (Proposition~\ref{prop:detblock}) to find the determinant. First, we find the Schur complement of $-(q+1)\mathbf{I}_t$ in $(\mathscr{D}_1 - M)(K_{s,t})$, which is given by
	\begin{equation*} 
		\begin{split}
			&\quad -q^2(q+1)\mathbf{J}_{s-1} - (q+1)\mathbf{I}_{s-1} + q(q+1)\mathbf{J}_{(s-1) \times t} \times \left((q+1)\mathbf{I}_{t} \right)^{-1} \times  q(q+1) \mathbf{J}_{t \times (s-1)}\\
			&=-q^2(q+1)\mathbf{J}_{s-1} - (q+1)\mathbf{I}_{s-1} +  q(q+1)\mathbf{J}_{(s-1) \times t} \times \frac{1}{q+1}\mathbf{I}_{t}  \times  q(q+1) \mathbf{J}_{t \times (s-1)}\\
			&=-q^2(q+1)\mathbf{J}_{s-1} - (q+1)\mathbf{I}_{s-1} + q^2(q+1)t \mathbf{J}_{s-1}\\
			&= q^2(q+1)(t-1)\mathbf{J}_{s-1} - (q+1)\mathbf{I}_{s-1}.
		\end{split}
	\end{equation*}
	Thus,	
	\begin{equation*} 
		\begin{split}
			\xi(\mathscr{D}(K_{s,t}))
			& = \det(-(q+1)\mathbf{I}_{t}) \times \det \left[q^2(q+1)(t-1)\mathbf{J}_{s-1} - (q+1)\mathbf{I}_{s-1}  \right]\\
			&= (-1)^t (q+1)^t \times (q+1)^{s-1} \det \left[q^2(t-1)\mathbf{J}_{s-1} -\mathbf{I}_{s-1}  \right]\\
			&= (-1)^t (q+1)^t \times (q+1)^{s-1} (-1)^{s-1} \left[q^2(s-1)(t-1)-1\right]\\
			&= (-1)^{s+t-1} (q+1)^{s+t-1} \left[q^2(s-1)(t-1)-1\right].
		\end{split}
	\end{equation*}
	Hence the result is proved.
\end{proof}

\begin{rem}
	Note that the cofactor at the $(1,1)$ position of the matrix $\mathscr{D}(K_{s,t})$ is nonzero if and only if $q \ne 1$ and $q^2(s-1)(t-1) \ne 1$, \textit{i.e.} the indeterminate $q$ satisfies Condition~\ref{con}.
\end{rem}

We complete the section with the expression for the inverse of the $q$-distance matrix of the complete bipartite graph $K_{s,t}$.

\begin{theorem}\label{thm:D-inv-single}
	Let $G$ be the complete bipartite graph $K_{s,t}$ and $\mathscr{D}(K_{s,t})$ be the $q$-distance matrix of $G$. Let $q$ be an indeterminate such that it satisfies Condition~\ref{con-1}, then 
	\begin{equation*} 
		\begin{split}
			\mathscr{D}(K_{s,t})^{-1} & = \frac{1}{(q+1)\left[(q+1)^2(s-1)(t-1)-st\right]}
			\begin{bmatrix}[cc]
				((q+1)^2(t-1)-t)\mathbf{J}_s & -(q+1)\mathbf{J}_{s \times t}\\
				-(q+1)\mathbf{J}_{t \times s} & ((q+1)^2(s-1)-s)\mathbf{J}_t
			\end{bmatrix}\\
			&\quad - \frac{1}{q+1} \mathbf{I}_{s+t}.
		\end{split}
	\end{equation*}
\end{theorem}

\begin{proof}
	The $q$-distance matrix of $K_{s,t}$ has the following block form:
	$$
	\mathscr{D}(K_{s,t}) = \begin{bmatrix}
		(q+1)(\mathbf{J}_{s} -\mathbf{I}_{s}) &  \mathbf{J}_{s\times t}\\
		\mathbf{J}_{t \times s} & (q+1)(\mathbf{J}_{t} - \mathbf{I}_{t})
	\end{bmatrix} = \begin{bmatrix}
	A &  B\\
	C & D
	\end{bmatrix} (say).
	$$
	We use the Schur complement technique (Proposition~\ref{prop:schur}) to find the inverse of $\mathscr{D}(K_{s,t})$. Since $A$ is invertible, the inverse is given by
	\[
	A^{-1} = \frac{1}{q+1} \left( \frac{1}{s-1} \mathbf{J}_s - \mathbf{I}_s\right)
	\]
	and from the calculations from Theorem~\ref{thm:D-det-single}, we have
	\[
	D-CA^{-1}B = \left(q+1-\frac{s}{(q+1)(s-1)}\right)\mathbf{J}_t - (q+1)\mathbf{I}_{t}.
	\]
	Next, using Lemma~\ref{lem:aI+bJ}, we have 
	\[
	(D-CA^{-1}B)^{-1} = \frac{(q+1)^2(s-1)-s}{(q+1)\left[(q+1)^2(s-1)(t-1)-st\right]}\mathbf{J}_t - \frac{1}{q+1}\mathbf{I}_t.
	\]
	Next, we compute
	\begin{equation*}
		\begin{split}
			(D-CA^{-1}B)^{-1} C A^{-1} &= (D-CA^{-1}B)^{-1} \times \mathbf{J}_{t \times s} \times \frac{1}{q+1} \left( \frac{1}{s-1} \mathbf{J}_s - \mathbf{I}_s\right)\\
			&= (D-CA^{-1}B)^{-1} \times \frac{1}{q+1} \left( \frac{s}{s-1} - 1\right)\mathbf{J}_{t \times s}\\
			&= \left[\frac{(q+1)^2(s-1)-s}{(q+1)\left[(q+1)^2(s-1)(t-1)-st\right]}\mathbf{J}_t - \frac{1}{q+1}\mathbf{I}_t\right] \times \frac{1}{(q+1)(s-1)} \mathbf{J}_{t \times s}\\
			&= \frac{1}{(q+1)^2(s-1)} \left[\frac{(q+1)^2(s-1)t-st}{(q+1)^2(s-1)(t-1)-st} - 1 \right]\mathbf{J}_{t \times s}\\
			&= \frac{1}{(q+1)^2(s-1)} \left[\frac{(q+1)^2(s-1)}{(q+1)^2(s-1)(t-1)-st} \right]\mathbf{J}_{t \times s}\\
			&= \frac{1}{(q+1)^2(s-1)(t-1)-st} \mathbf{J}_{t \times s}.
		\end{split}
	\end{equation*}
	Similarly, one can compute $A^{-1} B (D-CA^{-1}B)^{-1}$ and obtain 
	\[
	A^{-1} B (D-CA^{-1}B)^{-1} = \frac{1}{(q+1)^2(s-1)(t-1)-st} \mathbf{J}_{s \times t}.
	\]
	Finally using the above calculations we compute
	\begin{equation*}
		\begin{split}
			&A^{-1} + A^{-1} B (D-CA^{-1}B)^{-1} C A^{-1} \\
			&=  A^{-1} + \frac{1}{(q+1)^2(s-1)(t-1)-st} \mathbf{J}_{s \times t} \times \mathbf{J}_{t \times s} \times \frac{1}{q+1} \left( \frac{1}{s-1} \mathbf{J}_s - \mathbf{I}_s\right)\\
			&=  A^{-1} + \frac{1}{(q+1)^2(s-1)(t-1)-st} \times \frac{t}{(q+1)(s-1)} \mathbf{J}_s\\
			&= \frac{1}{q+1} \left( \frac{1}{s-1} \mathbf{J}_s - \mathbf{I}_s\right) + \frac{1}{(q+1)^2(s-1)(t-1)-st} \times \frac{t}{(q+1)(s-1)} \mathbf{J}_s\\
			&= \frac{t}{(q+1)(s-1)}  \left[ 1+ \frac{t}{(q+1)^2(s-1)(t-1)-st} \right]\mathbf{J}_s - \frac{1}{q+1}\mathbf{I}_s\\
			&= \frac{(q+1)^2(t-1)-t}{(q+1)\left[(q+1)^2(s-1)(t-1)-st\right]} \mathbf{J}_s - \frac{1}{q+1}\mathbf{I}_s.
		\end{split}
	\end{equation*}
	Therefore, by combining the above steps, the required result follows.
\end{proof}

\section{Some Useful Lemmas}
In this section, we present and establish several auxiliary lemmas that will serve as key tools for the proofs in the subsequent section.
\begin{lem}\label{lem:q+1}
	Let $G$ be a bi-block graph with $\mathbf{n}$ vertices and $\mathscr{D}$ be the $q$-distance matrix of $G$. Let $q$ be an indeterminate that satisfies Condition~\ref{con} and $\mathbf{x}$ be the $\mathbf{n}$-dimensional vector as defined in Equation~\ref{eqn:x-0}, then
	\[
	\sum_{i=1}^{\mathbf{n}} \left( 1+q+(q^2-1)\mathscr{D}_{i \mathbf{n}}\right)\mathbf{x}_i = q+1.
	\]
\end{lem}

\begin{proof}
	We use induction on the number of blocks, $r$ to complete the proof. Let $r=1$ and $G=K_{s,t}$, then $\mathbf{n} = s+t$. First we break the summation into two parts
	\begin{equation*}
		\begin{split}
			&\quad \sum_{i=1}^{\mathbf{n}} \left( 1+q+(q^2-1)\mathscr{D}_{i \mathbf{n}}\right)\mathbf{x}_i\\
			&= \sum_{i=1}^{s} \left( 1+q+(q^2-1)\mathscr{D}_{i \mathbf{n}}\right)\mathbf{x}_i + \sum_{i=s+1}^{s+t} \left( 1+q+(q^2-1)\mathscr{D}_{i \mathbf{n}}\right)\mathbf{x}_i.
		\end{split}
	\end{equation*}
	Next, substituting the respective expressions of $\mathbf{x}_i$, we have
	\begin{equation*}
		\begin{split}
			&\quad \sum_{i=1}^{\mathbf{n}} \left( 1+q+(q^2-1)\mathscr{D}_{i \mathbf{n}}\right)\mathbf{x}_i\\
			&= \sum_{i=1}^{s} \left( 1+q+(q^2-1)\right) \dfrac{q(t-1)-1}{(q+1)(q^2(s-1)(t-1)-1)} \\
			&\quad + \sum_{i=s+1}^{s+t-1} \left( 1+q+(q^2-1)(q+1)\right)\dfrac{q(s-1)-1}{(q+1)(q^2(s-1)(t-1)-1)} + (1+q)\dfrac{q(s-1)-1}{(q+1)(q^2(s-1)(t-1)-1)}\\
			&= \sum_{i=1}^{s} q(q+1) \dfrac{q(t-1)-1}{(q+1)(q^2(s-1)(t-1)-1)}\\
			&\quad + \sum_{i=s+1}^{s+t-1} q^2(q+1)\dfrac{q(s-1)-1}{(q+1)(q^2(s-1)(t-1)-1)}+ (1+q)\dfrac{q(s-1)-1}{(q+1)(q^2(s-1)(t-1)-1)}.
		\end{split}
	\end{equation*}
	Observe that, the terms inside the summation is independent of $i$, hence
		\begin{equation*}
		\begin{split}
			&\quad\sum_{i=1}^{\mathbf{n}} \left( 1+q+(q^2-1)\mathscr{D}_{i \mathbf{n}}\right)\mathbf{x}_i\\
			&= qs \dfrac{q(t-1)-1}{q^2(s-1)(t-1)-1} +  q^2(t-1)\dfrac{q(s-1)-1}{q^2(s-1)(t-1)-1}+\dfrac{q(s-1)-1}{q^2(s-1)(t-1)-1}\\
			&= \frac{qs(q(t-1)-1)+(q^2(t-1)+1)(q(s-1)-1)}{q^2(s-1)(t-1)-1}\\
			&= \frac{(q+1)(q^2(s-1)(t-1)-1)}{q^2(s-1)(t-1)-1} = q+1.
		\end{split}
	\end{equation*}
	Thus, the statement holds when $r=1$. We now proceed by induction: assume that the result is valid for every bi-block graph with $r-1$ blocks, and establish it for the case of $r$ blocks. Let $G$ be a bi-block graph consisting of $r$ blocks and $\mathbf{n}$ vertices. Denote by $K_{s,t}$ the leaf block of $G$, as introduced in Subsection~\ref{sec:ind}. Then,
	\begin{equation*}
		\begin{split}
			&\quad \sum_{i=1}^{\mathbf{n}} \left( 1+q+(q^2-1)\mathscr{D}_{i \mathbf{n}}\right)\mathbf{x}_i\\
			&= \sum_{i=1}^{\mathbf{m}-1} \left( 1+q+(q^2-1)\mathscr{D}_{i \mathbf{n}}\right)\mathbf{x}_i +\left( 1+q+(q^2-1)\mathscr{D}_{\mathbf{m} \mathbf{n}}\right)\mathbf{x}_{\mathbf{m}} + \sum_{i=\mathbf{m}+1}^{\mathbf{m}+s-1} \left( 1+q+(q^2-1)\mathscr{D}_{i \mathbf{n}}\right)\mathbf{x}_i\\
			&\quad + \sum_{i=\mathbf{m}+s}^{\mathbf{m}+s+t-2} \left( 1+q+(q^2-1)\mathscr{D}_{i \mathbf{n}}\right)\mathbf{x}_i + \left( 1+q+(q^2-1)\mathscr{D}_{\mathbf{n} \mathbf{n}}\right)\mathbf{x}_{\mathbf{n}}.
		\end{split}
	\end{equation*}
	Substituting the values of $\mathscr{D}_{i \mathbf{n}}$ for $i \ge \mathbf{m}$, we have
	\begin{equation*}
		\begin{split}
			&\sum_{i=1}^{\mathbf{n}} \left( 1+q+(q^2-1)\mathscr{D}_{i \mathbf{n}}\right)\mathbf{x}_i\\
			&= \sum_{i=1}^{\mathbf{m}-1} \left( 1+q+(q^2-1){\mathscr{D}}_{i \mathbf{n}}\right)\hat{\mathbf{x}}_i +\left( 1+q+(q^2-1)\right)\mathbf{x}_{\mathbf{m}} + \sum_{i=\mathbf{m}+1}^{\mathbf{m}+s-1} \left( 1+q+(q^2-1)\right)\mathbf{x}_i\\
			&\quad + \sum_{i=\mathbf{m}+s}^{\mathbf{m}+s+t-2} \left( 1+q+(q^2-1)(q+1)\right)\mathbf{x}_i + \left( 1+q\right)\mathbf{x}_{\mathbf{n}}\\
			&= \sum_{i=1}^{\mathbf{m}-1} \left( 1+q+(q^2-1){\mathscr{D}}_{i \mathbf{n}}\right)\hat{\mathbf{x}}_i + q(q+1)\mathbf{x}_{\mathbf{m}} + \sum_{i=\mathbf{m}+1}^{\mathbf{m}+s-1} q(q+1)\mathbf{x}_i + \sum_{i=\mathbf{m}+s}^{\mathbf{m}+s+t-2} q^2(q+1)\mathbf{x}_i + \left( 1+q\right)\mathbf{x}_{\mathbf{n}}\\
		\end{split}
	\end{equation*}
	Now using the respective expression for $\mathbf{x}_i$, from the second summand onward, we get
	\begin{equation*}
		\begin{split}
			&\sum_{i=1}^{\mathbf{n}} \left( 1+q+(q^2-1)\mathscr{D}_{i \mathbf{n}}\right)\mathbf{x}_i\\
			&= \sum_{i=1}^{\mathbf{m}-1} \left( 1+q+(q^2-1){\mathscr{D}}_{i \mathbf{n}}\right)\hat{\mathbf{x}}_i + q(q+1) \left( \hat{\mathbf{x}}_{\mathbf{m}} + \dfrac{q(t-1)-1}{(q+1)(q^2(s-1)(t-1)-1)} -1\right)\\
			&\quad +q(q+1)(s-1)\dfrac{q(t-1)-1}{(q+1)(q^2(s-1)(t-1)-1)} + q^2(q+1)(t-1) \dfrac{q(s-1)-1}{(q+1)(q^2(s-1)(t-1)-1)}\\
		\end{split}
	\end{equation*}
	Next, we separately deal with the terms that do not contain $\hat{\mathbf{x}}_i$ or $\hat{\mathbf{x}}_{\mathbf{m}}$.
	\begin{equation*}
		\begin{split}
			& q(q+1) \left( \dfrac{q(t-1)-1}{(q+1)(q^2(s-1)(t-1)-1)} -1\right) +q(q+1)(s-1)\dfrac{q(t-1)-1}{(q+1)(q^2(s-1)(t-1)-1)}\\
			& + q^2(q+1)(t-1) \dfrac{q(s-1)-1}{(q+1)(q^2(s-1)(t-1)-1)}\\
			&= qs\dfrac{q(t-1)-1}{q^2(s-1)(t-1)-1} +q^2(t-1)\dfrac{q(s-1)-1}{q^2(s-1)(t-1)-1} +\dfrac{q(s-1)-1}{q^2(s-1)(t-1)-1}-q(q+1)\\
			&= q+1 - q(q+1),
		\end{split}
	\end{equation*}
	where the simplification in last step follows from the computation for single block. Thus combining all the steps we have
	\begin{equation*}
		\begin{split}
			&\sum_{i=1}^{\mathbf{n}} \left( 1+q+(q^2-1)\mathscr{D}_{i \mathbf{n}}\right)\mathbf{x}_i\\
			&= \sum_{i=1}^{\mathbf{m}-1} \left( 1+q+(q^2-1){\mathscr{D}}_{i \mathbf{n}}\right)\hat{\mathbf{x}}_i + q(q+1)\hat{\mathbf{x}}_{\mathbf{m}} + q+1 - q(q+1)\\
			&= \sum_{i=1}^{\mathbf{m}-1} \left( 1+q+(q^2-1)(1+q\hat{\mathscr{D}}_{i \mathbf{m}})\right)\hat{\mathbf{x}}_i + q(q+1)\hat{\mathbf{x}}_{\mathbf{m}} + q+1 - q(q+1)\\
			&= q\sum_{i=1}^{\mathbf{m}} \left( 1+q+(q^2-1)\hat{\mathscr{D}}_{i \mathbf{m}}\right) \hat{\mathbf{x}}_i - q(q+1) + q+1 = + q+1,
		\end{split}
	\end{equation*}
	where, we have used the fact that $\hat{\mathscr{D}}_{\mathbf{m} \mathbf{m}} = 0$ and
	\[
	\sum_{i=1}^{\mathbf{m}} \left( 1+q+(q^2-1)\hat{\mathscr{D}}_{i \mathbf{m}}\right) \hat{\mathbf{x}}_i = q+1,
	\]
	which follows by applying induction hypothesis on $H$ and hence the result follows.
\end{proof}

\begin{cor}\label{cor:q+1}
Let $G$ be a bi-block graph with $\mathbf{n}$ vertices, and let $\mathscr{D}$ denote its $q$-distance matrix. If $\mathbf{x}$ is the $\mathbf{n}$-dimensional vector defined in Equation~\eqref{eqn:x-0}, then
\[
\sum_{i=1}^{\mathbf{n}} \left(1+(q-1)\mathscr{D}_{i \mathbf{n}}\right)\mathbf{x}_i = 1.
\]
\end{cor}

\noindent The proof follows immediately from Lemma~\ref{lem:q+1}, after factoring out $q+1$ from both sides of the equation. To proceed further, we introduce a notation that will be used in subsequent results. 
\begin{defn}\label{defn:lambdaG}
	Let $q$ be a non-zero indeterminate that satisfies Condition~\ref{con} and $G$ be the complete bipartite graph $K_{s,t}$ then we define
	\[
	\lambda_G = \frac{(q+1)^2(s-1)(t-1)-st}{(q+1)(q^2(s-1)(t-1)-1)}
	\]
	and if $G$ is a bi-block graph with $r$ blocks, $G_i = K_{m_i.n_i}$, for $1 \le i \le r$, then
	\begin{equation}\label{eqn:lambdaG}
		\lambda_G = \sum_{i=1}^r \lambda_{G_i}.
	\end{equation}
\end{defn}

\begin{lem}\label{lem:Dx=lamda}
	Let $q$ be a non-zero indeterminate that satisfies Condition~\ref{con} and $G$ be a bi-block graph with $\mathbf{n}$ vertices and $r$ blocks. Let $\mathscr{D}$ be the $q$-distance matrix of $G$ and $\mathbf{x}$ be the $\mathbf{n}$-dimensional vector as defined in Equation~\eqref{eqn:x-0}, then
	\[
	\mathscr{D} \mathbf{x} = \lambda_G \mathds{1}_{\mathbf{n}},
	\]
	where $\lambda_G$ is same as in Definition~\ref{defn:lambdaG}.
\end{lem}
\begin{proof}
	We complete the proof by induction on the number of blocks. For the base case $r=1$, when $G=K_{s,t}$, we have $\mathbf{n}=s+t$. The argument for a single block is divided into the following two cases:
		
	\underline{\textbf{Case 1.}} $1 \le i \le s$.
		\begin{equation*}
		\begin{split}
			\quad \sum_{j=1}^{\mathbf{n}} \mathscr{D}_{ij} \mathbf{x}_j &= (q+1)(s-1) \dfrac{q(t-1)-1}{(q+1)(q^2(s-1)(t-1)-1)} + t \dfrac{q(s-1)-1}{(q+1)(q^2(s-1)(t-1)-1)}\\
			&= \dfrac{(q+1)(s-1)(q(t-1)-1)+t(q(s-1)-1)}{(q+1)(q^2(s-1)(t-1)-1)}\\
			&= \frac{(q+1)^2(s-1)(t-1)-st}{(q+1)(q^2(s-1)(t-1)-1)}.
		\end{split}
	\end{equation*}
	
	\underline{\textbf{Case 2.}} $s+1 \le i \le s+t$.
	\begin{equation*}
		\begin{split}
			\quad \sum_{j=1}^{\mathbf{n}} \mathscr{D}_{ij} \mathbf{x}_j &= s \dfrac{q(t-1)-1}{(q+1)(q^2(s-1)(t-1)-1)} + (q+1)(t-1) \dfrac{q(s-1)-1}{(q+1)(q^2(s-1)(t-1)-1)}\\
			&= \dfrac{s(q(t-1)-1)+(q+1)(t-1)(q(s-1)-1)}{(q+1)(q^2(s-1)(t-1)-1)}\\
			&= \frac{(q+1)^2(s-1)(t-1)-st}{(q+1)(q^2(s-1)(t-1)-1)}.
		\end{split}
	\end{equation*}
Hence, the statement holds for $r=1$. For the inductive step, assume that the claim is valid for any bi-block graph with $r-1$ blocks; in particular, it holds for $H$, as defined in Subsection~\ref{sec:ind}, and we have 
$
\hat{\mathscr{D}} \hat{\mathbf{x}} = \lambda_H \mathds{1}_{\mathbf{m}}.
$
We now establish the result for $G$ with $r$ blocks, considering the following cases:

	\underline{\textbf{Case 1.}} $1 \le i \le \mathbf{m}$.
	\begin{equation*}
		\begin{split}
			\quad \sum_{j=1}^{\mathbf{n}} \mathscr{D}_{ij} \mathbf{x}_j 
			&= \sum_{j=1}^{\mathbf{m}-1} \mathscr{D}_{ij} \mathbf{x}_j + \mathscr{D}_{i \mathbf{m}} \mathbf{x}_{\mathbf{m}} + \sum_{j=\mathbf{m}+1}^{\mathbf{m}+s-1} \mathscr{D}_{ij} \mathbf{x}_j + \sum_{j=\mathbf{m}+s}^{\mathbf{m}+s+t-1} \mathscr{D}_{ij} \mathbf{x}_j\\
			&= \sum_{j=1}^{\mathbf{m}-1} \hat{\mathscr{D}}_{ij} \hat{\mathbf{x}}_j + \hat{\mathscr{D}}_{i \mathbf{m}} \left[\hat{\mathbf{x}}_{\mathbf{m}} + \dfrac{q(t-1)-1}{(q+1)(q^2(s-1)(t-1)-1)}-1\right]\\
			&\quad +  \sum_{j=\mathbf{m}+1}^{\mathbf{m}+s-1} \left(1+q+q^2\hat{\mathscr{D}}_{i \mathbf{m}}\right)\mathbf{x}_j + \sum_{j=\mathbf{m}+s}^{\mathbf{m}+s+t-1} \left(1+q\hat{\mathscr{D}}_{i \mathbf{m}}\right) \mathbf{x}_j.
		\end{split}
	\end{equation*}
Next, substituting the corresponding expressions for $\mathbf{x}_j$ in the third and fourth summands, we obtain
	\begin{equation*}
		\begin{split}
			\quad \sum_{j=1}^{\mathbf{n}} \mathscr{D}_{ij} \mathbf{x}_j 
			&= \sum_{j=1}^{\mathbf{m}} \hat{\mathscr{D}}_{ij} \hat{\mathbf{x}}_j +\hat{\mathscr{D}}_{i \mathbf{m}} \left[ \dfrac{q(t-1)-1}{(q+1)(q^2(s-1)(t-1)-1)}-1\right]\\
			&\quad + \left(1+q+q^2\hat{\mathscr{D}}_{i \mathbf{m}}\right) \dfrac{(s-1)(q(t-1)-1)}{(q+1)(q^2(s-1)(t-1)-1)}\\
			&\quad + \left(1+q\hat{\mathscr{D}}_{i \mathbf{m}}\right) \dfrac{t(q(s-1)-1)}{(q+1)(q^2(s-1)(t-1)-1)}.
		\end{split}
	\end{equation*}
	Now, grouping the terms that contains $\hat{\mathscr{D}}_{i \mathbf{m}}$, we get
	\begin{equation*}
		\begin{split}
			&\quad \left[\dfrac{q(t-1)-1}{(q+1)(q^2(s-1)(t-1)-1)}-1+ \dfrac{q^2(s-1)(q(t-1)-1)}{(q+1)(q^2(s-1)(t-1)-1)} + \dfrac{qt(q(s-1)-1)}{(q+1)(q^2(s-1)(t-1)-1)}\right] \hat{\mathscr{D}}_{i \mathbf{m}}\\
			&= \left[\dfrac{(q+1)(q^2(s-1)(t-1)-1)}{(q+1)(q^2(s-1)(t-1)-1)} -1\right]\hat{\mathscr{D}}_{i \mathbf{m}} = 0.
		\end{split}
	\end{equation*}
	Thus, we have
	\begin{equation*}
		\begin{split}
			\quad \sum_{j=1}^{\mathbf{n}} \mathscr{D}_{ij} \mathbf{x}_j 
			&= \sum_{j=1}^{\mathbf{m}} \hat{\mathscr{D}}_{ij} \hat{\mathbf{x}}_j + \dfrac{(q+1)(s-1)(q(t-1)-1)}{(q+1)(q^2(s-1)(t-1)-1)} + \dfrac{t(q(s-1)-1)}{(q+1)(q^2(s-1)(t-1)-1)}\\
			&= \sum_{j=1}^{\mathbf{m}} \hat{\mathscr{D}}_{ij} \hat{\mathbf{x}}_j + \frac{(q+1)^2(s-1)(t-1)-st}{(q+1)(q^2(s-1)(t-1)-1)} = \lambda_H + \frac{(q+1)^2(s-1)(t-1)-st}{(q+1)(q^2(s-1)(t-1)-1)} = \lambda_{G},
		\end{split}
	\end{equation*}
	where $\ds \sum_{j=1}^{\mathbf{m}} \hat{\mathscr{D}}_{ij} \hat{\mathbf{x}}_j = \lambda_H$ follows from the induction assumption.\\
	
	\underline{\textbf{Case 2.}} $\mathbf{m}+1 \le i \le \mathbf{m}+s-1$.
	\begin{equation*}
		\begin{split}
			\sum_{j=1}^{\mathbf{n}} \mathscr{D}_{ij} \mathbf{x}_j 
			&= \sum_{j=1}^{\mathbf{m}-1} \mathscr{D}_{ij} \mathbf{x}_j + \mathscr{D}_{i \mathbf{m}} \mathbf{x}_{\mathbf{m}} + \sum_{j=\mathbf{m}+1 \atop j \ne i}^{\mathbf{m}+s-1} \mathscr{D}_{ij} \mathbf{x}_j + \sum_{j=\mathbf{m}+s}^{\mathbf{m}+s+t-1} \mathscr{D}_{ij} \mathbf{x}_j\\
			&= \sum_{j=1}^{\mathbf{m}-1} \left(1+q+q^2\hat{\mathscr{D}}_{\mathbf{m} j}\right) \hat{\mathbf{x}}_j + \left(1+q+q^2\hat{\mathscr{D}}_{\mathbf{m} \mathbf{m}}\right)\mathbf{x}_{\mathbf{m}} + \sum_{j=\mathbf{m}+1 \atop j \ne i}^{\mathbf{m}+s-1} (1+q) \mathbf{x}_j + \sum_{j=\mathbf{m}+s}^{\mathbf{m}+s+t-1} \mathbf{x}_j.
		\end{split}
	\end{equation*}
	Next, using the respective expression for $\mathbf{x}_j$ in the second, third and fourth summand we have
	\begin{equation*}
		\begin{split}
			\sum_{j=1}^{\mathbf{n}} \mathscr{D}_{ij} \mathbf{x}_j 
			&= \sum_{j=1}^{\mathbf{m}-1} \left(1+q+q^2\hat{\mathscr{D}}_{\mathbf{m} j}\right) \hat{\mathbf{x}}_j + \left(1+q+q^2\hat{\mathscr{D}}_{\mathbf{m} \mathbf{m}}\right) \left[\hat{\mathbf{x}}_{\mathbf{m}} +\dfrac{q(t-1)-1}{(q+1)(q^2(s-1)(t-1)-1)}-1\right]\\
			&\quad +\sum_{j=\mathbf{m}+1 \atop j \ne i}^{\mathbf{m}+s-1} (1+q) \dfrac{q(t-1)-1}{(q+1)(q^2(s-1)(t-1)-1)} +  \sum_{j=\mathbf{m}+s}^{\mathbf{m}+s+t-1} \dfrac{q(s-1)-1}{(q+1)(q^2(s-1)(t-1)-1)}\\
			&= \sum_{j=1}^{\mathbf{m}} \left(1+q+q^2\hat{\mathscr{D}}_{\mathbf{m} j}\right) \hat{\mathbf{x}}_j + \left(1+q\right) \left[\dfrac{q(t-1)-1}{(q+1)(q^2(s-1)(t-1)-1)}-1\right]\\
			&\quad +\sum_{j=\mathbf{m}+1 \atop j \ne i}^{\mathbf{m}+s-1} (1+q) \dfrac{q(t-1)-1}{(q+1)(q^2(s-1)(t-1)-1)} +  \sum_{j=\mathbf{m}+s}^{\mathbf{m}+s+t-1} \dfrac{q(s-1)-1}{(q+1)(q^2(s-1)(t-1)-1)}.
		\end{split}
	\end{equation*}
	We now focus on the last three terms of the summation. Observe that in the third and fourth summands, the expressions inside the summation are independent of $j$. Hence, we obtain
	\begin{equation*}
		\begin{split}
			&\quad (q+1)\left[\dfrac{q(t-1)-1}{(q+1)(q^2(s-1)(t-1)-1)}-1\right] + (q+1)(s-2) \dfrac{q(t-1)-1}{(q+1)(q^2(s-1)(t-1)-1)}\\
			&\quad + t \dfrac{q(s-1)-1}{(q+1)(q^2(s-1)(t-1)-1)}\\
			&= \frac{(s-1)(q+1)(q(t-1)-1) + t(q(s-1)-1)}{(q+1)(q^2(s-1)(t-1)-1)}\\
			&=\frac{(q+1)^2(s-1)(t-1)-st}{(q+1)(q^2(s-1)(t-1)-1)}.
		\end{split}
	\end{equation*}
	Thus,
	\begin{equation*}
		\begin{split}
			\quad \sum_{j=1}^{\mathbf{n}} \mathscr{D}_{ij} \mathbf{x}_j 
			&= \sum_{j=1}^{\mathbf{m}} \left(1+q+q^2\hat{\mathscr{D}}_{\mathbf{m} j}\right) \hat{\mathbf{x}}_j - (q+1) + \frac{(q+1)^2(s-1)(t-1)-st}{(q+1)(q^2(s-1)(t-1)-1)}\\
			&= \sum_{j=1}^{\mathbf{m}} \hat{\mathscr{D}}_{\mathbf{m} j} \hat{\mathbf{x}}_j + \sum_{j=1}^{\mathbf{m}} \left(1+q+(q^2-1)\hat{\mathscr{D}}_{\mathbf{m} j}\right) \hat{\mathbf{x}}_j - (q+1) + \frac{(q+1)^2(s-1)(t-1)-st}{(q+1)(q^2(s-1)(t-1)-1)}\\
			&= \lambda_H + \frac{(q+1)^2(s-1)(t-1)-st}{(q+1)(q^2(s-1)(t-1)-1)} = \lambda_{G},
		\end{split}
	\end{equation*}
	where $\ds \sum_{j=1}^{\mathbf{m}} \left(1+q+(q^2-1)\hat{\mathscr{D}}_{\mathbf{m} j}\right) \hat{\mathbf{x}}_j - (q+1) = 0$ follows from Lemma~\ref{lem:q+1} and $\ds \sum_{j=1}^{\mathbf{m}} \hat{\mathscr{D}}_{ij} \hat{\mathbf{x}}_j = \lambda_H$ follows from the induction assumption.\\
	
	\underline{\textbf{Case 3.}} $\mathbf{m}+s \le i \le \mathbf{m}+s+t-1$.
	\begin{equation*}
		\begin{split}
			\quad \sum_{j=1}^{\mathbf{n}} \mathscr{D}_{ij} \mathbf{x}_j 
			&= \sum_{j=1}^{\mathbf{m}-1} \mathscr{D}_{ij} \mathbf{x}_j + \mathscr{D}_{i \mathbf{m}} \mathbf{x}_{\mathbf{m}} + \sum_{j=\mathbf{m}+1}^{\mathbf{m}+s-1} \mathscr{D}_{ij} \mathbf{x}_j + \sum_{j=\mathbf{m}+s \atop j \ne i}^{\mathbf{m}+s+t-1} \mathscr{D}_{ij} \mathbf{x}_j\\
			&= \sum_{j=1}^{\mathbf{m}-1} \left(1+q\hat{\mathscr{D}}_{\mathbf{m} j}\right) \hat{\mathbf{x}}_j + \left(1+q\hat{\mathscr{D}}_{\mathbf{m} \mathbf{m}}\right)\mathbf{x}_{\mathbf{m}} + \sum_{j=\mathbf{m}+1}^{\mathbf{m}+s-1} \mathbf{x}_j + \sum_{j=\mathbf{m}+s \atop j \ne i}^{\mathbf{m}+s+t-1} (1+q) \mathbf{x}_j.
		\end{split}
	\end{equation*}
	Now, we substitute the respective expression for $\mathbf{x}_j$ in the second, third and fourth summand we have
	\begin{equation*}
		\begin{split}
			\sum_{j=1}^{\mathbf{n}} \mathscr{D}_{ij} \mathbf{x}_j 
			&= \sum_{j=1}^{\mathbf{m}-1} \left(1+q\hat{\mathscr{D}}_{\mathbf{m} j}\right) \hat{\mathbf{x}}_j + \left(1+q\hat{\mathscr{D}}_{\mathbf{m} \mathbf{m}}\right)\left[\hat{\mathbf{x}}_{\mathbf{m}} +\dfrac{q(t-1)-1}{(q+1)(q^2(s-1)(t-1)-1)}-1\right]\\
			&\quad +\sum_{j=\mathbf{m}+1}^{\mathbf{m}+s-1} \dfrac{q(t-1)-1}{(q+1)(q^2(s-1)(t-1)-1)} + \sum_{j=\mathbf{m}+s \atop j \ne i}^{\mathbf{m}+s+t-1} (q+1) \dfrac{q(s-1)-1}{(q+1)(q^2(s-1)(t-1)-1)}\\
			&= \sum_{j=1}^{\mathbf{m}} \left(1+q\hat{\mathscr{D}}_{\mathbf{m} j}\right) \hat{\mathbf{x}}_j + \left[\dfrac{q(t-1)-1}{(q+1)(q^2(s-1)(t-1)-1)}-1\right]\\
			&\quad + \sum_{j=\mathbf{m}+1}^{\mathbf{m}+s-1} \dfrac{q(t-1)-1}{(q+1)(q^2(s-1)(t-1)-1)} + \sum_{j=\mathbf{m}+s \atop j \ne i}^{\mathbf{m}+s+t-1} (q+1) \dfrac{q(s-1)-1}{(q+1)(q^2(s-1)(t-1)-1)}.
		\end{split}
	\end{equation*}
	Next, we separate the last three terms of the summation and simplify them individually. Observe that, in the third and fourth summands, the expressions inside the summation are independent of $j$. Hence, we obtain
	\begin{equation*}
		\begin{split}
			&\dfrac{q(t-1)-1}{(q+1)(q^2(s-1)(t-1)-1)}-1 + \sum_{j=\mathbf{m}+1}^{\mathbf{m}+s-1} \dfrac{q(t-1)-1}{(q+1)(q^2(s-1)(t-1)-1)}\\
			&\quad + \sum_{j=\mathbf{m}+s \atop j \ne i}^{\mathbf{m}+s+t-1} (q+1) \dfrac{q(s-1)-1}{(q+1)(q^2(s-1)(t-1)-1)}\\
			&= \dfrac{s(q(t-1)-1)}{(q+1)(q^2(s-1)(t-1)-1)} + \dfrac{(t-1)(q+1)(q(s-1)-1)}{(q+1)(q^2(s-1)(t-1)-1)} -1\\
			&= \frac{(q+1)^2(s-1)(t-1)-st}{(q+1)(q^2(s-1)(t-1)-1)} -1.
		\end{split}
	\end{equation*}
	Hence,
	\begin{equation*}
		\begin{split}
			\quad \sum_{j=1}^{\mathbf{n}} \mathscr{D}_{ij} \mathbf{x}_j 
			&= \sum_{j=1}^{\mathbf{m}} \left(1+q\hat{\mathscr{D}}_{\mathbf{m} j}\right) \hat{\mathbf{x}}_j - 1 + \frac{(q+1)^2(s-1)(t-1)-st}{(q+1)(q^2(s-1)(t-1)-1)}\\
			&=\sum_{j=1}^{\mathbf{m}}\hat{\mathscr{D}}_{\mathbf{m} j} \hat{\mathbf{x}}_j  + \sum_{j=1}^{\mathbf{m}} \left(1+(q-1)\hat{\mathscr{D}}_{\mathbf{m} j}\right) \hat{\mathbf{x}}_j - 1 + \frac{(q+1)^2(s-1)(t-1)-st}{(q+1)(q^2(s-1)(t-1)-1)}\\
			&= \lambda_H + \frac{(q+1)^2(s-1)(t-1)-st}{(q+1)(q^2(s-1)(t-1)-1)} = \lambda_{G}, 
		\end{split}
	\end{equation*}
	where $\ds \sum_{j=1}^{\mathbf{m}} \left(1+(q-1)\hat{\mathscr{D}}_{\mathbf{m} j}\right) \hat{\mathbf{x}}_j - 1 = 0$ follows from Corollary~\ref{cor:q+1} and $\ds \sum_{j=1}^{\mathbf{m}} \hat{\mathscr{D}}_{ij} \hat{\mathbf{x}}_j = \lambda_H$ follows from the induction assumption. Thus, combining all the cases we have the required result.
\end{proof}

\begin{lem}\label{lem:sum-xi}
		Let $G$ be a bi-block graph with $\mathbf{n}$ vertices and $r$ blocks. Let $\mathscr{D}$ be the $q$-distance matrix of $G$ and $\mathbf{x}$ be the $\mathbf{n}$-dimensional vector as defined in Section~\ref{sec:notations}, then
		\begin{equation}\label{eqn:sum-xi}
			\sum_{i=1}^{\mathbf{n}} \mathbf{x}_i = 1 - (q-1) \lambda_{G}.
		\end{equation}
\end{lem}
\begin{proof}
	From Corollary~\ref{cor:q+1}, we have
	\[
	\sum_{i=1}^{\mathbf{n}} \left(1+(q-1)\mathscr{D}_{\mathbf{n} i}\right) \mathbf{x}_i = 1.
	\]
	Thus,
	\[
	\sum_{i=1}^{\mathbf{n}}  \mathbf{x}_i = 1 - (q-1) \sum_{i=1}^{\mathbf{n}} \mathscr{D}_{\mathbf{n} i}  \mathbf{x}_i
	\]
	and finally using Lemma~\ref{lem:Dx=lamda}, we have the required result.
\end{proof}

\section{Determinant and Inverse of Bi-block Graphs}
In this section, we derive the determinant and establish an explicit expression for the inverse of the $q$-distance matrix of bi-block graphs. We begin by stating the theorem for the determinant, without proof, as it follows directly from Theorem~\ref{thm:det-cof}.
\begin{theorem}
	Let $G$ be an \( \mathbf{n} \)-vertex bi-block graph with $r > 1$ blocks $K_{m_i,n_i}$, $i=1,2,\cdots,r$, where $\mathbf{n} = \sum_{i=1}^r (m_i+n_i)-r+1$ and $\mathscr{D}$ be the $q$-distance matrix of $G$. Let $q \ne -1$, then
	\[
	\xi(\mathscr{D}(G)) = \prod_{i=1}^{r} (-1)^{m_i+n_i-1} (q+1)^{m_i+n_i-1} \left[q^2(m_i-1)(n_i-1)-1\right]
	\]
	and
	\[
	\det(\mathscr{D}(G))  = \sum_{i=1}^r \det(\mathscr{D}(G_i)) \prod_{j \ne i} \xi(\mathscr{D}(G_j)),
	\]
	where $$\det(\mathscr{D}(G_i)) = (-1)^{m_i+n_i-2}(q+1)^{m_i+n_i-2} \left[(q+1)^2(m_i-1)(n_i-1)-m_in_i\right]$$ and $$\xi(\mathscr{D}(G_j)) = (-1)^{m_j+n_j-1} (q+1)^{m_j+n_j-1} \left[q^2(m_j-1)(n_j-1)-1\right].$$
\end{theorem}

\begin{rem}\label{rem:det-not0}
	Observe that, the determinant of the $q$-distance matrix of a bi-block graph is non-zero if and only if the indeterminate $q$ satisfies Condition~\ref{con-1}.
\end{rem}

\begin{defn}
	Let $G$ be an \( \mathbf{n} \)-vertex bi-block graph with $r > 1$ blocks $K_{m_i,n_i}$, $i=1,2,\cdots,r$, where $\mathbf{n} = \sum_{i=1}^r (m_i+n_i)-r+1$ and $\mathscr{D}$ be the $q$-distance matrix of $G$. Let $q$ be an indeterminate that satisfies Condition~\ref{con}, then we define an \( \mathbf{n} \times  \mathbf{n} \) matrix  $\mathscr{L}$ as follows:
	\begin{equation}\label{eqn:L}
		\mathscr{L} = \frac{q}{q+1} \mathbf{A} - \frac{q^2}{q+1} \mathbf{B} -  \frac{q^2}{q+1} \textup{diag}(\mathbf{y}) + \frac{1}{q+1} \mathbf{I},
	\end{equation}
	where the vector $\mathbf{y}$ and the matrices $\mathbf{A}$ and $\mathbf{B}$ are same as defined in Equations~\eqref{eqn:y-0}, \eqref{eqn:A} and \eqref{eqn:B}.
\end{defn}
Going by the same format, we will use $\hat{\mathscr{L}}$ to denote the corresponding matrix for the subgraph $H$. The main aim of this section is to give an expression for the inverse of the $q$-distance matrix of a bi-block graph in terms of $\mathscr{L}$, more explicitly, we what to show
\[
\mathscr{D}^{-1} = - \mathscr{L} + \frac{1}{\lambda_G} \mathbf{x} \mathbf{x}^T.
\] 

\begin{lem}\label{lem:JL}
	Let $G$ be an \( \mathbf{n} \)-vertex bi-block graph with $r > 1$ blocks and $\mathscr{D}$ be the $q$-distance matrix of $G$. If the indeterminate $q$ satisfies Conditions~\ref{con} and ~\ref{con-1} and $\mathscr{D}$ satisfies the relation $\mathscr{D} \mathscr{L} + \mathbf{I} = \mathds{1} \mathbf{x}^T$, then 
	\begin{equation}\label{eqn:JL-t}
		\mathbf{J}_{t \times \mathbf{n}} \hat{\mathscr{L}} = (1-q) \mathds{1}_t \hat{\mathbf{x}}^T
	\end{equation}
	and
	\begin{equation}\label{eqn:JL-s-1}
		\mathbf{J}_{(s-1) \times \mathbf{n}} \hat{\mathscr{L}} = (1-q) \mathds{1}_{s-1} \hat{\mathbf{x}}^T.
	\end{equation}
\end{lem}
\begin{proof}
	From the relation $\mathscr{D} \mathscr{L} + \mathbf{I} = \mathds{1}_\mathbf{n} \mathbf{x}^T$, we have $ \mathscr{L} = -\mathscr{D}^{-1} + \mathscr{D}^{-1} \mathds{1}_\mathbf{n} \mathbf{x}^T$. Now, using Lemma~\ref{lem:Dx=lamda}, we have 
	\begin{equation}\label{eqn:D-1-x}
		\mathscr{D}^{-1} \mathds{1}_\mathbf{n} = \frac{1}{\lambda_G} \mathbf{x}.
	\end{equation}
	Hence, the expression of $\mathscr{L}$ simplifies to $-\mathscr{D}^{-1} + \frac{1}{\lambda_G} \mathbf{x} \mathbf{x}^T$. Thus,
	\begin{equation*} 
		\begin{split}
			\mathbf{J}_{t \times \mathbf{n}} \hat{\mathscr{L}} 
			& = \mathbf{J}_{t \times \mathbf{n}} \left[-\mathscr{D}^{-1} + \frac{1}{\lambda_G} \mathbf{x} \mathbf{x}^T\right]\\
			&= - \mathds{1}_t \left(\mathds{1}_\mathbf{n}^T \mathscr{D}^{-1}\right)  + \frac{1}{\lambda_G} \mathds{1}_t \left(\mathds{1}_\mathbf{n}^T \mathbf{x}\right)  \mathbf{x}^T\\
			&\overset{\eqref{eqn:D-1-x}}{=}  - \frac{1}{\lambda_G} \mathds{1}_t \mathbf{x}^T + \frac{1}{\lambda_G} \mathds{1}_t \left(\sum_{i=1}^{\mathbf{n}} \mathbf{x}_{i}\right)  \mathbf{x}^T\\
			&\overset{\eqref{eqn:sum-xi}}{=}  - \frac{1}{\lambda_G} \mathds{1}_t \mathbf{x}^T + \frac{1}{\lambda_G} \mathds{1}_t \left(1 - (q-1) \lambda_{G}\right)  \mathbf{x}^T\\
			&= (1-q) \mathds{1}_t \hat{\mathbf{x}}^T.
		\end{split}
	\end{equation*}
	One can similarly prove the other identity.
\end{proof}

\begin{lem}\label{lem:E1-DL}
	Let $H$ be the subgraph of $G$, which satisfies the relation $\hat{\mathscr{D}} \hat{\mathscr{L}} + \mathbf{I} = \mathds{1}_{\mathbf{m}} \hat{\mathbf{x}}^T$, then
	\begin{equation}\label{eqn:E1-DL}
		E_1^T (\hat{\mathscr{D}} \hat{\mathscr{L}} + \mathbf{I} ) = \mathds{1}_{s-1} \hat{\mathbf{x}}^T
	\end{equation}
	and
	\begin{equation}\label{eqn:E1-Emm}
		E_1^T \hat{\mathscr{D}} E_{\mathbf{m} \mathbf{m}} = \mathbf{0}.
	\end{equation}
\end{lem}

\begin{proof}
	Since  $\hat{\mathscr{D}} \hat{\mathscr{L}} + \mathbf{I} = \mathds{1}_{\mathbf{m}} \hat{\mathbf{x}}^T$, we have 
	\[
	E_1^T (\hat{\mathscr{D}} \hat{\mathscr{L}} + \mathbf{I} ) = E_1^T \mathds{1}_{\mathbf{m}} \hat{\mathbf{x}}^T = \mathds{1}_{s-1} \left(e_{\mathbf{m}}^T \mathds{1}_{\mathbf{m}} \right) \hat{\mathbf{x}}^T = \mathds{1}_{s-1} \hat{\mathbf{x}}^T,
	\]
	where the last equality follows from the fact that $e_{\mathbf{m}}$ has only one non-zero entry $1$. Next,
	\[
	E_1^T \hat{\mathscr{D}} E_{\mathbf{m} \mathbf{m}} = (\mathds{1}_{s-1} e_{\mathbf{m}}^T) \hat{\mathscr{D}} (e_{\mathbf{m}} e_{\mathbf{m}}^T) = \mathds{1}_{s-1} \left(e_{\mathbf{m}}^T \hat{\mathscr{D}} e_{\mathbf{m}} \right) e_{\mathbf{m}}^T = \mathbf{0},
	\]
	where the last equality follows from the fact that the $(\mathbf{m},\mathbf{m})$ entry of $\mathscr{D}$ is $0$ and hence $e_{\mathbf{m}}^T \hat{\mathscr{D}} e_{\mathbf{m}} = 0$.
\end{proof}

\begin{rem}
	Under the same assumptions of Lemma~\ref{lem:E1-DL}, one can similarly prove that
	\begin{equation}\label{eqn:E2-DL}
		E_2^T (\hat{\mathscr{D}} \hat{\mathscr{L}} + \mathbf{I} ) = \mathds{1}_{t} \hat{\mathbf{x}}^T
	\end{equation}
	and
	\begin{equation}\label{eqn:E2-Emm}
		E_2^T \hat{\mathscr{D}} E_{\mathbf{m} \mathbf{m}} = \mathbf{0}.
	\end{equation}
\end{rem}

\begin{lem}\label{lem:LD+I-single}
	Let $G$ be the complete bipartite graph $K_{s,t}$ and $\mathscr{D}$ be the $q$-distance matrix of $G$. Let $q$ be an indeterminate that satisfies Condition~\ref{con-1}, then $$\mathscr{D} \mathscr{L} +\mathbf{I}_{s+t} = \mathds{1}_{s+t} \mathbf{x}^T.$$
\end{lem}
\begin{proof}
	From the definition of $\mathscr{L}$ in Equation~\ref{eqn:L}, we have the following expression for the complete bipartite graph $K_{s,t}$:
	\[
	\mathscr{L} = \frac{1}{(q+1)\left[q^2(s-1)(t-1)-1\right]} \begin{bmatrix}[cc]
		-q^2(t-1)\mathbf{J}_s & q\mathbf{J}_{s \times t}\\
		q\mathbf{J}_{t \times s} & -q^2(s-1)\mathbf{J}_t
	\end{bmatrix} + \frac{1}{q+1} \mathbf{I}_{s+t}.
	\]
	We complete the proof by computing each entry of the $2 \times 2$ block matrix $\mathscr{D} \mathscr{L} +\mathbf{I}_{s+t}$ and show it to be equal to the corresponding entry in $\mathds{1} \mathbf{x}^T$. First, we compute $(1,1)$ entry of $\mathscr{D} \mathscr{L} +\mathbf{I}_{s+t}$.
	\begin{equation*}
		\begin{split}
			(\mathscr{D} \mathscr{L} +\mathbf{I}_{s+t})_{11} &= \frac{-q^2(q+1)(s-1)(t-1)\mathbf{J}_s + qt\mathbf{J}_s}{(q+1)\left[q^2(s-1)(t-1)-1\right]} + (\mathbf{J}_s-\mathbf{I}_s) + \mathbf{I}_s\\
			&= \left[\frac{-q^2(q+1)(s-1)(t-1) + qt}{(q+1)\left[q^2(s-1)(t-1)-1\right]} + 1 \right]\mathbf{J}_s\\
			&= \frac{q(t-1)-1}{(q+1)\left[q^2(s-1)(t-1)-1\right]} \mathbf{J}_s.
		\end{split}
	\end{equation*}
	Next, we derive the $(1,2)$ entry of $\mathscr{D} \mathscr{L} +\mathbf{I}_{s+t}$.
	\begin{equation*}
		\begin{split}
			(\mathscr{D} \mathscr{L} +\mathbf{I}_{s+t})_{12} &= \frac{q(q+1)(s-1) -q^2(s-1)t}{(q+1)\left[q^2(s-1)(t-1)-1\right]}\mathbf{J}_{s \times t} + \frac{1}{q+1}\mathbf{J}_{s \times t}\\
			&= \left[\frac{q(q+1)(s-1) -q^2(s-1)t}{(q+1)\left[q^2(s-1)(t-1)-1\right]} + \frac{1}{q+1} \right]\mathbf{J}_{s\times t}\\
			&= \frac{q(s-1)-1}{(q+1)\left[q^2(s-1)(t-1)-1\right]} \mathbf{J}_{s \times t}.
		\end{split}
	\end{equation*}
	Now, we evaluate the $(2,1)$ entry of $\mathscr{D} \mathscr{L} +\mathbf{I}_{s+t}$.
	\begin{equation*}
		\begin{split}
			(\mathscr{D} \mathscr{L} +\mathbf{I}_{s+t})_{21} &= \frac{-q^2s(t-1)+q(q+1)(t-1)}{(q+1)\left[q^2(s-1)(t-1)-1\right]}\mathbf{J}_{t \times s} + \frac{1}{q+1}\mathbf{J}_{t \times s}\\
			&= \left[\frac{-q^2s(t-1)+q(q+1)(t-1)}{(q+1)\left[q^2(s-1)(t-1)-1\right]} + \frac{1}{q+1} \right]\mathbf{J}_{s\times t}\\
			&= \frac{q(t-1)-1}{(q+1)\left[q^2(s-1)(t-1)-1\right]} \mathbf{J}_{s \times t}.
		\end{split}
	\end{equation*}
	Finally, we compute the $(2,2)$ entry of $\mathscr{D} \mathscr{L} +\mathbf{I}_{s+t}$.
	\begin{equation*}
		\begin{split}
			(\mathscr{D} \mathscr{L} +\mathbf{I}_{s+t})_{22} &= \frac{qs-q^2(q+1)(s-1)(t-1)}{(q+1)\left[q^2(s-1)(t-1)-1\right]}\mathbf{J}_{t} +(\mathbf{J}_t-\mathbf{I}_t) + \mathbf{I}_t\\
			&= \left[\frac{qs-q^2(q+1)(s-1)(t-1)}{(q+1)\left[q^2(s-1)(t-1)-1\right]} + 1\right]\mathbf{J}_{t}\\
			&= \frac{q(s-1)-1}{(q+1)\left[q^2(s-1)(t-1)-1\right]} \mathbf{J}_{ t}.
		\end{split}
	\end{equation*}
	Note that, the matrix $\mathds{1} \mathbf{x}^T$ can be written in the following block form:
	\[
	\mathds{1} \mathbf{x}^T = \frac{1}{(q+1)\left[q^2(s-1)(t-1)-1\right]} \begin{bmatrix}[cc]
		(q(t-1)-1)\mathbf{J}_s & (q(s-1)-1)\mathbf{J}_{s \times t}\\
		(q(t-1)-1)\mathbf{J}_{t \times s} & (q(s-1)-1)\mathbf{J}_t
	\end{bmatrix}
	\]
	and hence comparing both the matrices term by term the result follows.
\end{proof}

In the next theorem we prove the identity $\mathscr{D} \mathscr{L} +\mathbf{I}_{\mathbf{n}} = \mathds{1}_{\mathbf{n}} \mathbf{x}^T$ for any bi-block graph with at least $2$ blocks.
\begin{theorem}\label{thm:DL+I}
	Let $G$ be an \( \mathbf{n} \)-vertex bi-block graph with $r > 1$ blocks $K_{m_i,n_i}$, $i=1,2,\cdots,r$, where $\mathbf{n} = \sum_{i=1}^r (m_i+n_i)-r+1$ and $\mathscr{D}$ be the $q$-distance matrix of $G$. Let $q$ be an indeterminate that satisfies Conditions~\ref{con} and \ref{con-1}, then $$\mathscr{D} \mathscr{L} +\mathbf{I}_{\mathbf{n}} = \mathds{1}_{\mathbf{n}} \mathbf{x}^T.$$
\end{theorem}
\begin{proof}
	We establish the result by induction on the number of blocks $r$. For the base case $r=1$, when $G=K_{s,t}$, the claim follows directly from Lemma~\ref{lem:LD+I-single}. For the inductive step, assume that the statement holds for all bi-block graphs with $r-1$ blocks; in particular, it is valid for $H$, \textit{i.e.},
	\begin{equation}\label{eqn:LD+I-H}
		\hat{\mathscr{D}} \hat{\mathscr{L}} +\mathbf{I}_{\mathbf{m}} = \mathds{1}_{\mathbf{m}} \hat{\mathbf{x}}^T.
	\end{equation}
	We partition the matrix $\mathscr{D} \mathscr{L} +\mathbf{I}_{\mathbf{n}}$ into a $3 \times 3$ block matrix, conformally to the partition of $\mathscr{D}$ and $\mathscr{L}$ and compute each of the entries to prove the result. We start by finding the $(1,1)$ entry of  $\mathscr{D} \mathscr{L} +\mathbf{I}_{\mathbf{n}}$.
	\begin{equation*} 
		\begin{split}
			&\quad (\mathscr{D} \mathscr{L} +\mathbf{I}_{\mathbf{n}})_{11}\\
			&= \frac{q}{q+1} (\mathscr{D}\mathbf{A})_{11} - \frac{q^2}{q+1} (\mathscr{D}\mathbf{B})_{11} -  \frac{q^2}{q+1} (\mathscr{D}\textup{diag}(\mathbf{y}))_{11} + \frac{1}{q+1} \mathscr{D}_{11} + \mathbf{I}_{\mathbf{m}}\\
			&= \frac{q}{q+1} \left[\hat{\mathscr{D}} \hat{\mathbf{A}} + \frac{(\mathbf{J}_{\mathbf{m} \times t} + q \hat{\mathscr{D}} E_2) E_2^T}{q^2(s-1)(t-1) - 1} \right] - \frac{q^2}{q+1} \left[\hat{\mathscr{D}} \hat{\mathbf{B}} + \frac{(t-1)((q+1)\mathbf{J}_{\mathbf{m} \times (s-1)} + q^2\hat{\mathscr{D}} E_1) E_1^T}{q^2(s-1)(t-1) - 1}\right]\\
			&\quad - \frac{q^2}{q+1} \left[ \hat{\mathscr{D}} \textup{diag}(\hat{\mathbf{y}}) + \left(\frac{t-1}{q^2(s-1)(t-1)-1} - 1\right) \hat{\mathscr{D}} E_{\mathbf{m}\mathbf{m}}\right] + \frac{1}{q+1} \hat{\mathscr{D}} + \mathbf{I}_{\mathbf{m}}\\
			&= \frac{q}{q+1} \left[\hat{\mathscr{D}} \hat{\mathbf{A}} + \frac{t E_0^T + tq \hat{\mathscr{D}} E_{\mathbf{m}\mathbf{m}}}{q^2(s-1)(t-1) - 1} \right] - \frac{q^2}{q+1} \left[\hat{\mathscr{D}} \hat{\mathbf{B}} + \frac{(t-1)((q+1)(s-1)E_0^T + q^2(s-1) \hat{\mathscr{D}}E_{\mathbf{m}\mathbf{m}}) }{q^2(s-1)(t-1) - 1}\right]\\
			&\quad - \frac{q^2}{q+1} \left[ \hat{\mathscr{D}} \textup{diag}(\hat{\mathbf{y}}) + \left(\frac{t-1}{q^2(s-1)(t-1)-1} - 1\right) \hat{\mathscr{D}} E_{\mathbf{m}\mathbf{m}}\right] + \frac{1}{q+1} \hat{\mathscr{D}} + \mathbf{I}_{\mathbf{m}}.
		\end{split}
	\end{equation*}
	Now, collecting all the terms containing $\hat{\mathscr{D}}$, we have 
	\[
	\frac{q}{q+1} \hat{\mathscr{D}} \hat{\mathbf{A}} - \frac{q^2}{q+1} \hat{\mathscr{D}} \hat{\mathbf{B}}- \frac{q^2}{q+1} \hat{\mathscr{D}} \textup{diag}(\hat{\mathbf{y}}) + \frac{1}{q+1} \hat{\mathscr{D}} \overset{\eqref{eqn:L}}{=}  \hat{\mathscr{D}} \hat{\mathscr{L}}.
	\]
	Let us focus on the two terms containing $E_0^T$, which leads to
	\begin{equation*} 
		\begin{split}
			&\quad \left[\frac{q}{q+1} \cdot \frac{t}{q^2(s-1)(t-1)-1} - \frac{q^2}{q+1} \cdot \frac{(q+1)(t-1)(s-1)}{q^2(s-1)(t-1)-1} \right] E_0^T \\
			&= \left[\frac{q(t-1)-1}{(q+1)(q^2(s-1)(t-1)-1)} -1\right]E_0^T.
		\end{split}
	\end{equation*}
	Finally, grouping all the terms that contains $\hat{\mathscr{D}} E_{\mathbf{m}\mathbf{m}}$, we have
	\begin{equation*} 
		\begin{split}
			&\quad \frac{q^2}{q+1} \left[\frac{t-q^2(s-1)(t-1)-t+1}{q^2(s-1)(t-1)-1} +1 \right]\hat{\mathscr{D}} E_{\mathbf{m}\mathbf{m}} \\
			&= \frac{q^2}{q+1} \left[\frac{1-q^2(s-1)(t-1)}{q^2(s-1)(t-1)-1} +1 \right]\hat{\mathscr{D}} E_{\mathbf{m}\mathbf{m}} = \mathbf{0}.
		\end{split}
	\end{equation*}
	Thus, combining all of the above expressions we have
	\begin{equation*} 
		\begin{split}
			\quad (\mathscr{D} \mathscr{L} +\mathbf{I}_{\mathbf{n}})_{11}
			&= \hat{\mathscr{D}} \hat{\mathscr{L}} +\mathbf{I}_{\mathbf{m}} + \left[\frac{q(t-1)-1}{(q+1)(q^2(s-1)(t-1)-1)} -1\right]E_0^T\\
			&\overset{\eqref{eqn:LD+I-H}}{=} \mathds{1}_{\mathbf{m}} \hat{\mathbf{x}}^T + \left[\frac{q(t-1)-1}{(q+1)(q^2(s-1)(t-1)-1)} -1\right]E_0^T.
		\end{split}
	\end{equation*}
	Next, we evaluate the $(1,2)$ entry of $\mathscr{D} \mathscr{L} +\mathbf{I}_{\mathbf{n}}$.
	\begin{equation*} 
		\begin{split}
			&\quad (\mathscr{D} \mathscr{L} +\mathbf{I}_{\mathbf{n}})_{12}\\
			&= \frac{q}{q+1} (\mathscr{D}\mathbf{A})_{12} - \frac{q^2}{q+1} (\mathscr{D}\mathbf{B})_{12}-  \frac{q^2}{q+1} (\mathscr{D}\textup{diag}(\mathbf{y}))_{12} + \frac{1}{q+1} \mathscr{D}_{12}\\
			&= \frac{q}{q+1} \left[ \frac{t \mathbf{J}_{\mathbf{m} \times (s-1)} + tq \hat{\mathscr{D}} E_1}{q^2(s-1)(t-1)-1}\right]\\
			&\quad  - \frac{q^2}{q+1} \left[ \frac{(t-1)\hat{\mathscr{D}} E_1 + (q+1)(s-2)(t-1) \mathbf{J}_{\mathbf{m} \times (s-1)} + q^2 (s-2)(t-1)\hat{\mathscr{D}}E_1}{q^2(s-1)(t-1)-1}\right]\\
			&\quad - \frac{q^2}{q+1} \left[\frac{(q+1)(t-1)\mathbf{J}_{\mathbf{m} \times (s-1)} +q^2(t-1)\hat{\mathscr{D}}E_1}{q^2(s-1)(t-1)-1}\right] + \frac{1}{q+1} \left[(q+1)\mathbf{J}_{\mathbf{m} \times (s-1)} + q^2\hat{\mathscr{D}}E_1 \right].
		\end{split}
	\end{equation*}
	Let us group the terms containing $\hat{\mathscr{D}}E_1$, which leads to
	\begin{equation*} 
		\begin{split}
			&\quad \left[\frac{q}{q+1} \cdot \frac{qt}{q^2(s-1)(t-1)-1} - \frac{q^2}{q+1} \cdot \frac{(t-1) + q^2(s-2)(t-1)}{q^2(s-1)(t-1)-1} \right. \\
			&\quad \qquad \left. - \frac{q^2}{q+1} \cdot \frac{q^2(t-1)}{q^2(s-1)(t-1)-1} + \frac{q^2}{q+1} \right]\hat{\mathscr{D}}E_1\\
			&= \frac{q^2}{q+1} \left[\frac{t-(t-1)-q^2(s-2)(t-1)-q^2(t-1)}{q^2(s-1)(t-1)-1} +1 \right]\hat{\mathscr{D}}E_1\\
			&= \frac{q^2}{q+1} \left[\frac{1-q^2(s-1)(t-1)}{q^2(s-1)(t-1)-1} +1 \right]\hat{\mathscr{D}}E_1 = \mathbf{0}.
		\end{split}
	\end{equation*}
	Next, focusing on all the term containing $ \mathbf{J}_{\mathbf{m} \times (s-1)}$, we have
	\begin{equation*} 
		\begin{split}
			&\quad \left[\frac{q}{q+1} \cdot \frac{t}{q^2(s-1)(t-1)-1} - \frac{q^2}{q+1} \cdot \frac{(q+1)(s-2)(t-1)}{q^2(s-1)(t-1)-1} \right. \\
			&\quad \qquad \left. - \frac{q^2}{q+1} \cdot \frac{(q+1)(t-1)}{q^2(s-1)(t-1)-1} +1\right]\mathbf{J}_{\mathbf{m} \times (s-1)}\\
			&= \left[\frac{qt}{(q+1)(q^2(s-1)(t-1)-1)} - \frac{1}{q^2(s-1)(t-1)-1}\right]\mathbf{J}_{\mathbf{m} \times (s-1)}\\
			&= \frac{q(t-1)-1}{(q+1)(q^2(s-1)(t-1)-1)} \mathbf{J}_{\mathbf{m} \times (s-1)}.
		\end{split}
	\end{equation*}
	Thus, we have 
	\[
	(\mathscr{D} \mathscr{L} +\mathbf{I}_{\mathbf{n}})_{12} = \frac{q(t-1)-1}{(q+1)(q^2(s-1)(t-1)-1)} \mathbf{J}_{\mathbf{m} \times (s-1)}.
	\]\\
	At this point, we proceed to determine the $(1,3)$ entry $\mathscr{D} \mathscr{L} +\mathbf{I}_{\mathbf{n}}$.
	\begin{equation*} 
		\begin{split}
			(\mathscr{D} \mathscr{L} +\mathbf{I}_{\mathbf{n}})_{13}
			&= \frac{q}{q+1} (\mathscr{D}\mathbf{A})_{13} - \frac{q^2}{q+1} (\mathscr{D}\mathbf{B})_{13} -  \frac{q^2}{q+1} (\mathscr{D}\textup{diag}(\mathbf{y}))_{13} + \frac{1}{q+1} \mathscr{D}_{13}\\
			&= \frac{q}{q+1} \left[\frac{\hat{\mathscr{D}}E_2 + (q+1)(s-1)\mathbf{J}_{\mathbf{m} \times t} +q^2 (s-1)\hat{\mathscr{D}}E_2}{q^2(s-1)(t-1)-1}\right]\\
			&\quad - \frac{q^2}{q+1} \left[\frac{(s-1)(t-1)\mathbf{J}_{\mathbf{m} \times t} + q(s-1)(t-1)\hat{\mathscr{D}}E_2}{q^2(s-1)(t-1)-1}\right]\\
			&\quad - \frac{q^2}{q+1} \left[\frac{(s-1)\mathbf{J}_{\mathbf{m} \times t} + q(s-1)\hat{\mathscr{D}}E_2}{q^2(s-1)(t-1)-1}\right] + \frac{1}{q+1} \left[\mathbf{J}_{\mathbf{m} \times t}+q\hat{\mathscr{D}}E_2\right].
		\end{split}
	\end{equation*}
	First, we collect all the terms that contain $ \mathbf{J}_{\mathbf{m} \times t}$, which leads to
	\begin{equation*} 
		\begin{split}
			&\quad \left[\frac{q}{q+1} \cdot \frac{(q+1)(s-1)}{q^2(s-1)(t-1)-1} - \frac{q^2}{q+1} \cdot \frac{(s-1)(t-1)}{q^2(s-1)(t-1)-1} \right. \\
			&\quad \qquad \left. - \frac{q^2}{q+1} \cdot\frac{s-1}{q^2(s-1)(t-1)-1}+ \frac{1}{q+1} \right]\mathbf{J}_{\mathbf{m} \times t}\\
			&= \frac{1}{q+1} \left[\frac{q(q+1)(s-1) -q^2(t-1)(s-1) -q^2(s-1)}{q^2(s-1)(t-1)-1} +1\right]\mathbf{J}_{\mathbf{m} \times t}\\
			&= \frac{q(s-1)-1}{(q+1)(q^2(s-1)(t-1)-1)} \mathbf{J}_{\mathbf{m} \times t}.
		\end{split}
	\end{equation*}
	Next, we collect all the terms that contain $\hat{\mathscr{D}}E_2$ and we have
	\begin{equation*} 
		\begin{split}
			&\quad \left[\frac{q}{q+1} \cdot \frac{1+q^2(s-1)}{q^2(s-1)(t-1)-1} - \frac{q^2}{q+1} \cdot \frac{q(s-1)(t-1)}{q^2(s-1)(t-1)-1}  \right. \\
			&\quad \qquad \left. - \frac{q^2}{q+1} \cdot \frac{q(s-1)}{q^2(s-1)(t-1)-1}  + \frac{q}{q+1}\right]\hat{\mathscr{D}}E_2\\
			&= \frac{q}{q+1} \left[\frac{q^2(s-1)+1-q^2(s-1)(t-1)-q^2(s-1)}{q^2(s-1)(t-1)-1} +1\right]\hat{\mathscr{D}}E_2\\
			&= \frac{q}{q+1} \left[\frac{1-q^2(s-1)(t-1)}{q^2(s-1)(t-1)-1} +1\right]\hat{\mathscr{D}}E_2 = \mathbf{0}.
		\end{split}
	\end{equation*}
	Thus, we have
	\[
	(\mathscr{D} \mathscr{L} +\mathbf{I}_{\mathbf{n}})_{13} = \frac{q(s-1)-1}{(q+1)(q^2(s-1)(t-1)-1)} \mathbf{J}_{\mathbf{m} \times t}.
	\]\\
	We now carry out the computation of the $(2,1)$ entry $\mathscr{D} \mathscr{L} +\mathbf{I}_{\mathbf{n}}$.
	\begin{equation*} 
		\begin{split}
			(\mathscr{D} \mathscr{L} +\mathbf{I}_{\mathbf{n}})_{21}
			&= \frac{q}{q+1} (\mathscr{D}\mathbf{A})_{21} - \frac{q^2}{q+1} (\mathscr{D}\mathbf{B})_{21} -  \frac{q^2}{q+1} (\mathscr{D}\textup{diag}(\mathbf{y}))_{21} + \frac{1}{q+1} \mathscr{D}_{21}\\
			&= \left[(q+1) \mathbf{J}_{(s-1) \times \mathbf{m}} +q^2 E_1^T \hat{\mathscr{D}}\right] \cdot \left[\frac{q}{q+1}\hat{\mathbf{A}} - \frac{q^2}{q+1} \hat{\mathbf{B}}- \frac{q^2}{q+1}\textup{diag}(\hat{\mathbf{y}}) \right.  \\
			&\qquad \left. - \frac{q^2}{q+1} \left(\dfrac{t-1}{q^2(s-1)(t-1)-1} -1\right)E_{\mathbf{m} \mathbf{m}}  + \frac{1}{q+1} \mathbf{I}_{\mathbf{m}}\right]\\
			&\quad +\frac{q}{q+1} \cdot \frac{\mathbf{J}_{(s-1) \times t} E_2^T}{q^2(s-1)(t-1)-1} - \frac{q^2}{q+1} \cdot \frac{(q+1)(t-1)(\mathbf{J}_{s-1} - \mathbf{I}_{s-1}) E_1^T}{q^2(s-1)(t-1)-1}.
		\end{split}
	\end{equation*}
	Simplifying we get 
	\begin{equation*} 
		\begin{split}
			(\mathscr{D} \mathscr{L} +\mathbf{I}_{\mathbf{n}})_{21}
			&=\left[(q+1) \mathbf{J}_{(s-1) \times \mathbf{m}} +q^2 E_1^T \hat{\mathscr{D}}\right] \cdot \left[\frac{q}{q+1}\hat{\mathbf{A}} - \frac{q^2}{q+1} \hat{\mathbf{B}}- \frac{q^2}{q+1}\textup{diag}(\hat{\mathbf{y}}) \right.  \\
			&\qquad \left. - \frac{q^2}{q+1} \left(\dfrac{t-1}{q^2(s-1)(t-1)-1} -1\right)E_{\mathbf{m} \mathbf{m}}  + \frac{1}{q+1} \mathbf{I}_{\mathbf{m}}\right]\\
			&\quad +\frac{q}{q+1} \cdot \frac{t E_1^T}{q^2(s-1)(t-1)-1} - \frac{q^2}{q+1} \cdot \frac{(q+1)(s-2)(t-1)E_1^T}{q^2(s-1)(t-1)-1}.
		\end{split}
	\end{equation*}
	Now, using the fact that $\mathbf{J}_{(s-1) \times \mathbf{m}}E_{\mathbf{m} \mathbf{m}} = E_1^T$ from Remark~\ref{rem:identities} and collecting the terms containing only $E_1^T$, we have
	\begin{equation*} 
		\begin{split}
			&\quad \left[\frac{q}{q+1} \cdot \frac{t}{q^2(s-1)(t-1)-1} - \frac{q^2}{q+1} \cdot \frac{(q+1)(s-2)(t-1)}{q^2(s-1)(t-1)-1}  \right. \\
			&\quad \qquad \left. - q^2 \left(\dfrac{t-1}{q^2(s-1)(t-1)-1} -1\right) \right]E_1^T\\
			&= \left[\frac{qt-q^2(q+1)(s-2)(t-1) -q^2(q+1)(t-1)}{(q+1)(q^2(s-1)(t-1)-1)} \right]E_1^T + q^2E_1^T\\
			&= \left[\frac{q(t-1) -1}{(q+1)(q^2(s-1)(t-1)-1)} -1 \right]E_1^T + q^2E_1^T
		\end{split}
	\end{equation*}
	Let us focus on the remaining terms that contain $\mathbf{J}_{(s-1) \times \mathbf{m}}$. This leads to
	\begin{equation*} 
	(q+1) \mathbf{J}_{(s-1) \times \mathbf{m}} \left[ \frac{q}{q+1}\hat{\mathbf{A}} - \frac{q^2}{q+1} \hat{\mathbf{B}}- \frac{q^2}{q+1}\textup{diag}(\hat{\mathbf{y}}) + \frac{1}{q+1} \mathbf{I}_{\mathbf{m}}\right] \overset{\eqref{eqn:L}}{=}  (q+1) \mathbf{J}_{(s-1) \times \mathbf{m}} \hat{\mathscr{L}}.
	\end{equation*}
	Then, using Lemma~\ref{lem:JL}, we have $(q+1) \mathbf{J}_{(s-1) \times \mathbf{m}} \hat{\mathscr{L}} = - (q^2-1) \mathds{1}_{s-1} \hat{\mathbf{x}}^T.$ Further, using Lemma~\ref{lem:E1-DL} and grouping the remaining terms of $E_1^T \hat{\mathscr{D}}$, we have
	\[
	q^2 E_1^T \hat{\mathscr{D}} \left[ \frac{q}{q+1}\hat{\mathbf{A}} - \frac{q^2}{q+1} \hat{\mathbf{B}}- \frac{q^2}{q+1}\textup{diag}(\hat{\mathbf{y}}) + \frac{1}{q+1} \mathbf{I}_{\mathbf{m}}\right] \overset{\eqref{eqn:L}}{=} q^2 E_1^T \hat{\mathscr{D}} \hat{\mathscr{L}}.
	\]
	Thus, combining all the parts above we have
	\begin{equation*} 
		\begin{split}
			&\quad (\mathscr{D} \mathscr{L} +\mathbf{I}_{\mathbf{n}})_{21} \\
			&= \left[\frac{q(t-1) -1}{(q+1)(q^2(s-1)(t-1)-1)} -1 \right]E_1^T + q^2E_1^T - (q^2-1) \mathds{1}_{s-1} \hat{\mathbf{x}}^T + q^2 E_1^T \hat{\mathscr{D}} \hat{\mathscr{L}}\\
			&= \left[\frac{q(t-1) -1}{(q+1)(q^2(s-1)(t-1)-1)} -1 \right]E_1^T + q^2E_1^T \left(\mathbf{I}_{\mathbf{m}} + \hat{\mathscr{D}} \hat{\mathscr{L}}\right)  - (q^2-1) \mathds{1}_{s-1} \hat{\mathbf{x}}^T\\
			&\overset{\eqref{eqn:E1-DL}}{=}\left[\frac{q(t-1) -1}{(q+1)(q^2(s-1)(t-1)-1)} -1 \right]E_1^T + q^2 \mathds{1}_{s-1} \hat{\mathbf{x}}^T - (q^2-1) \mathds{1}_{s-1} \hat{\mathbf{x}}^T\\
			&= \left[\frac{q(t-1) -1}{(q+1)(q^2(s-1)(t-1)-1)} -1 \right]E_1^T +  \mathds{1}_{s-1} \hat{\mathbf{x}}^T.
		\end{split}
	\end{equation*}\\
	Next, we evaluate the $(2,2)$ entry $\mathscr{D} \mathscr{L} +\mathbf{I}_{\mathbf{n}}$.
	\begin{equation*} 
		\begin{split}
			(\mathscr{D} \mathscr{L} +\mathbf{I}_{\mathbf{n}})_{22}
			&= \frac{q}{q+1} (\mathscr{D}\mathbf{A})_{22} - \frac{q^2}{q+1} (\mathscr{D}\mathbf{B})_{22} -  \frac{q^2}{q+1} (\mathscr{D}\textup{diag}(\mathbf{y}))_{22} + \frac{1}{q+1} \mathscr{D}_{22} + \mathbf{I}_{s-1}\\
			&=\frac{q}{q+1} \cdot \frac{t \mathbf{J}_{s-1}}{q^2(s-1)(t-1)-1} -  \frac{q^2}{q+1} \cdot \frac{(q+1)(t-1)\left((s-2) \mathbf{J}_{s-1} + \mathbf{I}_{s-1}\right)}{q^2(s-1)(t-1)-1}\\
			&\quad - \frac{q^2}{q+1} \cdot \frac{(q+1)(t-1)\left(\mathbf{J}_{s-1} - \mathbf{I}_{s-1}\right)}{q^2(s-1)(t-1)-1} + \mathbf{J}_{s-1}.\\
		\end{split}
	\end{equation*}
	Note that the terms containing $\mathbf{I}_{s-1}$ cancel each other. Now, let us group the terms containing $\mathbf{J}_{s-1}$, which leads to
	\begin{equation*} 
		\begin{split}
			& \left[\frac{q}{q+1} \cdot \frac{t}{q^2(s-1)(t-1)-1} -  \frac{q^2}{q+1} \cdot \frac{(q+1)(t-1)(s-2)}{q^2(s-1)(t-1)-1}\right. \\
			&\quad \left. - \frac{q^2}{q+1} \cdot \frac{(q+1)(t-1)}{q^2(s-1)(t-1)-1} + 1 \right]\mathbf{J}_{s-1}\\
			&= \frac{q(t-1) -1}{(q+1)(q^2(s-1)(t-1)-1)}\mathbf{J}_{s-1}.
		\end{split}
	\end{equation*}
	Thus, we have
	\[
	(\mathscr{D} \mathscr{L} +\mathbf{I}_{\mathbf{n}})_{22} = \frac{q(t-1) -1}{(q+1)(q^2(s-1)(t-1)-1)} \mathbf{J}_{s-1}.
	\]
	We now find the $(2,3)$ entry $\mathscr{D} \mathscr{L} +\mathbf{I}_{\mathbf{n}}$.
	\begin{equation*} 
		\begin{split}
			(\mathscr{D} \mathscr{L} +\mathbf{I}_{\mathbf{n}})_{23}
			&= \frac{q}{q+1} (\mathscr{D}\mathbf{A})_{23} - \frac{q^2}{q+1} (\mathscr{D}\mathbf{B})_{23} -  \frac{q^2}{q+1} (\mathscr{D}\textup{diag}(\mathbf{y}))_{23} + \frac{1}{q+1} \mathscr{D}_{23}\\
			&=\frac{q}{q+1} \cdot \frac{(q+1)(s-2)}{q^2(s-1)(t-1)-1} \mathbf{J}_{(s-1) \times t} -  \frac{q^2}{q+1} \cdot \frac{(s-1)(t-1)}{q^2(s-1)(t-1)-1}\mathbf{J}_{(s-1) \times t} \\
			&\quad - \frac{q^2}{q+1} \cdot \frac{s-1}{q^2(s-1)(t-1)-1}\mathbf{J}_{(s-1) \times t} + \frac{1}{q+1} \mathbf{J}_{(s-1) \times t}\\
			&= \frac{1}{q+1} \left[\frac{q(q+1)(s-2) -q^2(s-1)(t-1) -q^2(s-1)}{q^2(s-1)(t-1)-1} + 1\right] \mathbf{J}_{(s-1) \times t}\\
			&= \frac{q(t-1) -1}{(q+1)(q^2(s-1)(t-1)-1)} \mathbf{J}_{(s-1) \times t}.
		\end{split}
	\end{equation*}
	At this point, we proceed to determine the $(3,1)$ entry $\mathscr{D} \mathscr{L} +\mathbf{I}_{\mathbf{n}}$.
	\begin{equation*} 
		\begin{split}
			(\mathscr{D} \mathscr{L} +\mathbf{I}_{\mathbf{n}})_{31}
			&= \frac{q}{q+1} (\mathscr{D}\mathbf{A})_{31} - \frac{q^2}{q+1} (\mathscr{D}\mathbf{B})_{31} -  \frac{q^2}{q+1} (\mathscr{D}\textup{diag}(\mathbf{y}))_{31} + \frac{1}{q+1} \mathscr{D}_{31}\\
			&= \left[\mathbf{J}_{t \times \mathbf{m}} +q E_2^T \hat{\mathscr{D}}\right] \cdot \left[\frac{q}{q+1}\hat{\mathbf{A}} - \frac{q^2}{q+1} \hat{\mathbf{B}}- \frac{q^2}{q+1}\textup{diag}(\hat{\mathbf{y}}) \right.  \\
			&\qquad \left. - \frac{q^2}{q+1} \left(\dfrac{t-1}{q^2(s-1)(t-1)-1} -1\right)E_{\mathbf{m} \mathbf{m}}  + \frac{1}{q+1} \mathbf{I}_{\mathbf{m}}\right]\\
			&\quad +\frac{q}{q+1} \cdot \frac{(q+1)(t-1)}{q^2(s-1)(t-1)-1} E_2^T- \frac{q^2}{q+1} \cdot \frac{(s-1)(t-1)}{q^2(s-1)(t-1)-1} E_2^T.
		\end{split}
	\end{equation*}
	Now, using the fact that $\mathbf{J}_{t \times \mathbf{m}}E_{\mathbf{m} \mathbf{m}} = E_2^T$ from Remark~\ref{rem:identities} and collecting the terms containing only $E_2^T$, we have
	\begin{equation*} 
		\begin{split}
			&\quad \left[\frac{q}{q+1} \cdot \frac{(q+1)(t-1)}{q^2(s-1)(t-1)-1} - \frac{q^2}{q+1} \cdot \frac{(s-1)(t-1)}{q^2(s-1)(t-1)-1}  \right. \\
			&\quad \qquad \left. - \frac{q^2}{q+1} \left(\dfrac{t-1}{q^2(s-1)(t-1)-1} -1\right) \right]E_2^T\\
			&= \left[\frac{q(q+1)(t-1)-q^2(s-1)(t-1) -q^2(t-1)}{(q+1)(q^2(s-1)(t-1)-1)} + \frac{1}{q+1} \right]E_2^T + \frac{q^2-1}{q+1}E_2^T\\
			&= \left[\frac{q(t-1) -1}{(q+1)(q^2(s-1)(t-1)-1)} -1 \right]E_2^T + qE_2^T.
		\end{split}
	\end{equation*}
	Now, let us focus on the remaining terms that contain $\mathbf{J}_{t \times \mathbf{m}}$, which leads to
	\begin{equation*} 
		\mathbf{J}_{t \times \mathbf{m}} \left[ \frac{q}{q+1}\hat{\mathbf{A}} - \frac{q^2}{q+1} \hat{\mathbf{B}}- \frac{q^2}{q+1}\textup{diag}(\hat{\mathbf{y}}) + \frac{1}{q+1} \mathbf{I}_{\mathbf{m}}\right] \overset{\eqref{eqn:L}}{=}  \mathbf{J}_{t \times \mathbf{m}} \hat{\mathscr{L}}.
	\end{equation*}
	Then, using Lemma~\ref{lem:JL}, we have $\mathbf{J}_{t \times \mathbf{m}} \hat{\mathscr{L}} = (1-q) \mathds{1}_{t} \hat{\mathbf{x}}^T.$ Further, using Lemma~\ref{lem:E1-DL} and grouping the remaining terms of $E_2^T \hat{\mathscr{D}}$, we have
	\[
	q E_2^T \hat{\mathscr{D}} \left[ \frac{q}{q+1}\hat{\mathbf{A}} - \frac{q^2}{q+1} \hat{\mathbf{B}}- \frac{q^2}{q+1}\textup{diag}(\hat{\mathbf{y}}) + \frac{1}{q+1} \mathbf{I}_{\mathbf{m}}\right] \overset{\eqref{eqn:L}}{=} q E_2^T \hat{\mathscr{D}} \hat{\mathscr{L}}.
	\]
	Thus, combining all the parts above we have
	\begin{equation*} 
		\begin{split}
			&\quad (\mathscr{D} \mathscr{L} +\mathbf{I}_{\mathbf{n}})_{31} \\
			&= \left[\frac{q(t-1) -1}{(q+1)(q^2(s-1)(t-1)-1)} -1 \right]E_2^T + qE_2^T - (q-1) \mathds{1}_{t} \hat{\mathbf{x}}^T + q E_2^T \hat{\mathscr{D}} \hat{\mathscr{L}}\\
			&= \left[\frac{q(t-1) -1}{(q+1)(q^2(s-1)(t-1)-1)} -1 \right]E_2^T + qE_2^T \left(\mathbf{I}_{\mathbf{m}} + \hat{\mathscr{D}} \hat{\mathscr{L}}\right)  - (q-1) \mathds{1}_{t} \hat{\mathbf{x}}^T\\
			&\overset{\eqref{eqn:E2-DL}}{=}\left[\frac{q(t-1) -1}{(q+1)(q^2(s-1)(t-1)-1)} -1 \right]E_2^T + q^2 \mathds{1}_{t} \hat{\mathbf{x}}^T - (q-1) \mathds{1}_{t} \hat{\mathbf{x}}^T\\
			&= \left[\frac{q(t-1) -1}{(q+1)(q^2(s-1)(t-1)-1)} -1 \right]E_2^T +  \mathds{1}_{t} \hat{\mathbf{x}}^T.
		\end{split}
	\end{equation*}
	Subsequently, we compute the $(3,2)$ entry $\mathscr{D} \mathscr{L} +\mathbf{I}_{\mathbf{n}}$.
	\begin{equation*} 
		\begin{split}
			(\mathscr{D} \mathscr{L} +\mathbf{I}_{\mathbf{n}})_{32}
			&= \frac{q}{q+1} (\mathscr{D}\mathbf{A})_{32} - \frac{q^2}{q+1} (\mathscr{D}\mathbf{B})_{32} -  \frac{q^2}{q+1} (\mathscr{D}\textup{diag}(\mathbf{y}))_{32} + \frac{1}{q+1} \mathscr{D}_{32}\\
			&=\frac{q}{q+1} \cdot \frac{(q+1)(t-1)}{q^2(s-1)(t-1)-1} \mathbf{J}_{t \times (s-1)} -  \frac{q^2}{q+1} \cdot \frac{(s-1)(t-1)}{q^2(s-1)(t-1)-1} \mathbf{J}_{t \times (s-1)} \\
			&\quad - \frac{q^2}{q+1} \cdot \frac{t-1}{q^2(s-1)(t-1)-1} \mathbf{J}_{t \times (s-1)} + \frac{1}{q+1} \mathbf{J}_{t \times (s-1)}\\
			&= \frac{1}{q+1} \left[\frac{q(q+1)(t-1) -q^2(s-1)(t-1) -q^2(t-1)}{q^2(s-1)(t-1)-1} + 1\right] \mathbf{J}_{t \times (s-1)}\\
			&= \frac{q(t-1) -1}{(q+1)(q^2(s-1)(t-1)-1)} \mathbf{J}_{t \times (s-1)}.
		\end{split}
	\end{equation*}
	At this stage, it remains to compute the $(3,3)$ entry $\mathscr{D} \mathscr{L} +\mathbf{I}_{\mathbf{n}}$.
	\begin{equation*} 
		\begin{split}
			(\mathscr{D} \mathscr{L} +\mathbf{I}_{\mathbf{n}})_{33}
			&= \frac{q}{q+1} (\mathscr{D}\mathbf{A})_{33} - \frac{q^2}{q+1} (\mathscr{D}\mathbf{B})_{33} -  \frac{q^2}{q+1} (\mathscr{D}\textup{diag}(\mathbf{y}))_{33} + \frac{1}{q+1} \mathscr{D}_{33} + \mathbf{I}_{t}\\
			&=\frac{q}{q+1} \cdot \frac{s \mathbf{J}_{t}}{q^2(s-1)(t-1)-1} -  \frac{q^2}{q+1} \cdot \frac{(q+1)(s-1)\left((t-2) \mathbf{J}_{t} + \mathbf{I}_{t}\right)}{q^2(s-1)(t-1)-1}\\
			&\quad - \frac{q^2}{q+1} \cdot \frac{(q+1)(s-1)\left(\mathbf{J}_{t} - \mathbf{I}_{t}\right)}{q^2(s-1)(t-1)-1} + \mathbf{J}_{t}.\\
		\end{split}
	\end{equation*}
	Note that the terms containing $\mathbf{I}_{t}$ cancel each other. Now, let us group the terms containing $\mathbf{J}_{t}$, which leads to
	\begin{equation*} 
		\begin{split}
			& \left[\frac{q}{q+1} \cdot \frac{s}{q^2(s-1)(t-1)-1} -  \frac{q^2}{q+1} \cdot \frac{(q+1)(s-1)(t-2)}{q^2(s-1)(t-1)-1}\right. \\
			&\quad \left. - \frac{q^2}{q+1} \cdot \frac{(q+1)(s-1)}{q^2(s-1)(t-1)-1} + 1 \right]\mathbf{J}_{t}\\
			&= \frac{q(s-1) -1}{(q+1)(q^2(s-1)(t-1)-1)}\mathbf{J}_{t}.
		\end{split}
	\end{equation*}
	Thus, we have
	\[
	(\mathscr{D} \mathscr{L} +\mathbf{I}_{\mathbf{n}})_{33} = \frac{q(s-1) -1}{(q+1)(q^2(s-1)(t-1)-1)} \mathbf{J}_{t}.
	\]
	Finally comparing all the entries, we have $\mathscr{D} \mathscr{L} +\mathbf{I}_{\mathbf{n}} = \mathds{1}_{\mathbf{n}} \mathbf{x}^T.$ Hence the result follows by induction.
\end{proof}

\begin{theorem}
	Let $G$ be an \( \mathbf{n} \)-vertex bi-block graph with $r > 1$ blocks $K_{m_i,n_i}$, $i=1,2,\cdots,r$, where $\mathbf{n} = \sum_{i=1}^r (m_i+n_i)-r+1$ and $\mathscr{D}$ be the $q$-distance matrix of $G$. Let $q$ be an indeterminate that satisfies Conditions~\ref{con} and \ref{con-1}, then
	\[
	\mathscr{D}^{-1} = - \mathscr{L} + \frac{1}{\lambda_G} \mathbf{x} \mathbf{x}^T.
	\]
\end{theorem}

\begin{proof}
	Using Theorem~\ref{thm:DL+I}, we have that $\mathscr{D} \mathscr{L} +\mathbf{I}_{\mathbf{n}} = \mathds{1}_{\mathbf{n}} \mathbf{x}^T.$ Since the indeterminate $q$ satisfies Condition~\ref{con-1}, it follows from Remark~\ref{rem:det-not0} that $\mathscr{D}$ is invertible hence we have
	\[
	\mathscr{L} + \mathscr{D}^{-1} =  \mathscr{D}^{-1} \mathds{1}_{\mathbf{n}} \mathbf{x}^T.
	\]
	Now using Lemma~\ref{lem:Dx=lamda}, we have $\mathscr{D}^{-1} \mathds{1}_{\mathbf{n}} = \frac{1}{\lambda_{G}} \mathbf{x}$. Thus, we have
	\[
	\mathscr{D}^{-1} = - \mathscr{L} + \frac{1}{\lambda_G} \mathbf{x} \mathbf{x}^T
	\]
	and the result follows.
\end{proof}

\section*{Acknowledgments}
The author, Joyentanuj Das, gratefully acknowledges Iswar Mahato for introducing the problem, which arose during their discussions.

\small{

}


\begin{thebibliography}{20}
\bibitem{Ba-Kr-Neu} R.B. Bapat, S.J. Kirkland, M. Neumann, On distance matrices and Laplacians, Linear Algebra Appl. 401: 193–209, 2005.

\bibitem{Bapat0} R.B. Bapat, A.K. Lal, S. Pati, A q-analogue of the distance matrix of a tree, Linear Algebra Appl. 416: 799–814, 2006.

\bibitem{Bapat1} R.B. Bapat, Determinant of the distance matrix of a tree with matrix weights, Linear Algebra Appl. 416: 2–7, 2006.

\bibitem{Ba-Lal-Pati} R.B. Bapat, A.K. Lal, S. Pati, The distance matrix of a bidirected tree, Electron. J. Linear Algebra 18: 233–245, 2009.

\bibitem{Bapat}  R.B. Bapat.  Graphs and matrices. Second Edition, Hindustan Book Agency, 
New Delhi. 2014.

\bibitem{Bp3}  R.B. Bapat and  S. Sivasubramanian. Inverse of the distance matrix of a block graph. Linear and Multilinear Algebra.  59(12):1393-1397, 2011. 

\bibitem{Barik} S. Barik, M. Mondal, and S. Pan. A q-analogue of the distance matrix of a tree with matrix weights. Linear and Multilinear Algebra, 73(7), 1423–1447, 2024. 

\bibitem{JD}J. Das, S. Jayaraman and  S. Mohanty. Distance Matrix of a Class of Completely Positive Graphs: Determinant and Inverse. Spec. Matrices 8:160-171, 2020. 

\bibitem{JD1} J. Das and  S. Mohanty, Distance matrix of multi-block graphs: determinant and inverse. Linear and Multilinear Algebra, 70(19), 3994–4022, 2020.

\bibitem{JD2} J. Das and S. Mohanty, Distance matrix of weighted cactoid-type digraphs, Linear and Multilinear Algebra, 70(20), 5392–5422, 2021.

\bibitem{Edelberg} M. Edelberg, M.R. Garey, R.L. Graham, On the distance matrix of a tree, Discrete Math. 14 (1976) 23–29.

\bibitem{Gr1}  R.L. Graham and  H.O. Pollak. On the addressing problem for loop switching.  Bell System Technical Journal.  50:2495-2519, 1971.

\bibitem{Gr3} R. L. Graham, A. J. Hoffman and H. Hosoya. On the distance matrix of a directed graph. J. Graph Theory. 1:85-88, 1977.

\bibitem{Gr2} R.L. Graham and  L. Lov\'asz. Distance matrix polynomials of trees. Advances in  Mathematics. 29(1):60-88,  1978.

\bibitem{Hou1} Y. Hou and J. Chen. Inverse of the distance matrix of a cactoid digraph. Linear Algebra and its Applications.  475:1-10, 2015.

\bibitem{Hou2} Y. Hou, A. Fang and Y. Sun. Inverse of the distance matrix of a cycle-clique graph. Linear Algebra and its Applications. 485:33-46, 2015.

\bibitem{Hou3} Y. Hou and Y. Sun. Inverse of the distance matrix of a bi-block graph. Linear and Multilinear Algebra.  64(8):1509-1517, 2016.

\bibitem{Li} H.-H. Li, L. Su, J. Zhang, On the determinant of q-distance matrix of a graph, Discuss. Math., Graph Theory 34 (2014) 103–111.

\bibitem{Siva} S. Sivasubramanian, q-Analogs of distance matrices of 3-hypertrees, Linear Algebra Appl. 431: 1234–1248, 2009.

\bibitem{Siva1}  S. Sivasubramanian, A q-analogue of Graham, Hoffman and Hosoya’s theorem, Electron. J. Comb. 17, 2010.

\bibitem{Xing} R. Xing and Z. Du, A $q$-analogue of distance matrix of block graphs, Advances in Applied Mathematics, Volume 149, 2023, 102548.

\bibitem{Yan}  W. Yan, Y.N. Yeh, The determinants of q-distance matrices of trees and two quantities relating to permutations, Adv. Appl. Math. 39 (2007) 311–321.

\bibitem{Zhang} F. Zhang. Matrix theory. Basic results and Techniques. Universitext, Springer Verlag, New York. 1999.

\bibitem{Zhang1} F. Zhang (Editor). The Schur complement and its applications. Numerical Methods and Algorithms, 4. Springer-Verlag, New York, 2005.

\bibitem{Zhou1} H. Zhou. The inverse of the distance matrix of a distance well-defined graph.
Linear Algebra  and its Applications. 517:11-29, 2017. 

\bibitem{Zhou2} H. Zhou, Q. Ding and R. Jia. Inverse of the distance matrix of a weighted cactoid digraph. Applied Mathematics and Computation 362, 2019. 





\end{thebibliography}
\end{document}